\documentclass[review,onefignum,onetabnum]{siamart171218}
\usepackage{amsmath}
\usepackage{amssymb}
\usepackage{enumerate}
\usepackage{graphicx}
\graphicspath{{figures/}}
\setkeys{Gin}{draft=false}
\usepackage{algorithm,algorithmic}
\usepackage{caption}
\usepackage{array}
\usepackage{pgfplots, pgfplotstable, booktabs, colortbl, array}
\usepackage{multirow}
\usepackage{placeins}
\usepackage{bm}

\newcommand{\FEP}{\textsc{FEP}\xspace}
\newcommand{\FIP}{\textsc{FIP}\xspace}
\newcommand{\FIPSG}{\textsc{FIP-SG}\xspace}
\newcommand{\FIPSGEBC}{\textsc{FIP-SG-EBC}\xspace}
\newcommand{\FIPSGEFC}{\textsc{FIP-SG-EFC}\xspace}

\usepackage{float}
\floatstyle{plaintop}
\restylefloat{table}
\usepackage{todonotes}

\newcommand{\ignore}[1]{}

\usepackage{xspace}

\usepackage[leqno]{amsmath}
\usepackage{amsfonts}
\usepackage{amssymb}
\newcommand{\norm}[1]{\|#1\|}
\newcommand{\abs}[1]{|#1|}

\nolinenumbers


\usepackage{lipsum}
\usepackage{amsfonts}
\usepackage{graphicx}
\usepackage{epstopdf}
\usepackage{algorithmic}
\ifpdf
  \DeclareGraphicsExtensions{.eps,.pdf,.png,.jpg}
\else
  \DeclareGraphicsExtensions{.eps}
\fi

\newsiamremark{remark}{Remark}
\newsiamremark{hypothesis}{Hypothesis}
\crefname{hypothesis}{Hypothesis}{Hypotheses}
\newsiamthm{claim}{Claim}

\headers{Numerical schemes for 1D CCP}{Chen, Fan, Fan, Guan, Liu, Qiao, Wei, You and Zou}

\title{Numerical schemes for 1D capacitively coupled plasma fluid simulation\thanks{Submitted to the editors DATE.
\funding{The work of Zhonghua Qiao is partially sponsored by NSFC/RGC Joint Research Scheme grant N\_PolyU5145/24 and Hong Kong Research Grants Council GRF grant 15305624.}}}

\author{
Yipei Chen\thanks{Theory Lab, Huawei Leibniz Research Center, Sha Tin, Hong Kong SAR, People's Republic of China. (\email{chen.yipei1@huawei.com}, 
\email{fanganghua1@huawei.com},
\email{fanyuwei2@huawei.com}, \email{guan.haoran@huawei.com}, \email{youpeiting@huawei.com}},
\email{liukangxin3@huawei.com})~
\and
Ganghua Fan\footnotemark[2]~
\and
Yuwei Fan\footnotemark[2]~
\and
Haoran Guan\footnotemark[2]~
\and
Kangxin Liu\footnotemark[2]~
\and
Zhonghua Qiao\thanks{Department of Applied Mathematics, Hong Kong Polytechnic University, Hung Hom, Hong Kong SAR, People's Republic of China. (\email{zhonghua.qiao@polyu.edu.hk})}~
\and
Yufeng Wei\thanks{Corresponding author. Theory Lab, Huawei Leibniz Research Center, Sha Tin, Hong Kong SAR, People's Republic of China. (\email{wei.yufeng@huawei.com})}~
\and
Peiting You\footnotemark[2]~
\and
Zhenyu Zou\thanks{Department of Mechanical and Automation Engineering, The Chinese University of Hong Kong, Sha Tin, Hong Kong SAR, People's Republic of China. (\email{zyzou@mae.cuhk.edu.hk})}
}

\usepackage{amsopn}

\makeatletter
\newcommand*{\addFileDependency}[1]{% argument=file name and extension
  \typeout{(#1)}% latexmk will find this if $recorder=0 (however, in that case, it will ignore #1 if it is a .aux or .pdf file etc and it exists! if it doesn't exist, it will appear in the list of dependents regardless)
  \@addtofilelist{#1}% if you want it to appear in \listfiles, not really necessary and latexmk doesn't use this
  \IfFileExists{#1}{}{\typeout{No file #1.}}% latexmk will find this message if #1 doesn't exist (yet)
}
\makeatother

\newcommand*{\myexternaldocument}[1]{%
    \externaldocument{#1}%
    \addFileDependency{#1.tex}%
    \addFileDependency{#1.aux}%
}
\ifpdf
\hypersetup{
  pdftitle={A new numerical scheme for 1D capacitively coupled plasma fluid simulation},
  pdfauthor={Yipei Chen, Ganghua Fan, Yuwei Fan, Haoran Guan, Kangxin Liu, Zhonghua Qiao, Yufeng Wei, Peiting You, Zhenyu Zou}
}
\fi

\myexternaldocument{ex_supplement}

\begin{document}

\maketitle

% REQUIRED
\begin{abstract}
Capacitively coupled plasmas (CCP) are widely used in semiconductor manufacturing, where fluid models provide an efficient approach to the simulation of the low-temperature discharge. The model considered in this work couples the electron and the ion continuity equations, the electron energy equation and the Poisson's equation for the electric potential. We investigate the numerical schemes for the simulation of the 1D-CCP. In particular, a semi-implicit Scharfetter-Gummel (SG) scheme is adopted for the Poisson's equation, while the remaining equations are solved using the implicit SG schemes. The proposed strategy can be naturally extended to the 2D cases. Several alternative schemes are also discussed and tested. Numerical experiments compare these methods in terms of the admissible time step $\Delta t$, total simulation time, accuracy against a reference solution and robustness of the scheme. We propose a numerical scheme for the 1D-CCP and, to the best of our knowledge, provide the first detailed numerical analysis of such a scheme for CCP.
\end{abstract}

% REQUIRED

% REQUIRED

\section{Introduction}

\subsection{Model and background knowledge}
Capacitively coupled plasmas (CCPs) are among the representative radio-frequency low-temperature discharges which are widely used in semiconductor manufacturing, including plasma etching, thin-film deposition and surface treatment \cite{AlvesMarques2012,Bi2011}. They also serve as valuable model systems for their simple reactor configuration and rich dynamics to study sheath evolution, electron heating, power deposition and plasma uniformity \cite{BoeufPitchford1995,Zhang2012a,Zhang2012b}.

Fluid models are widely used for low-temperature plasma simulations because they resolve the charged-particle transport, electrostatic potential and the electron energy at a substantially lower cost than the kinetic models. Therefore, they have become an important tool for analyzing the radio-frequency discharges \cite{AlvesMarques2012, Belenguer1990, Boeuf1987, ChenRaja2004, Kim2005, Surendra1990}. In particular, the configuration of the 1D-CCP shown in \cref{fig:bound} retains the essential coupling among charged-particle transport, electron energy and the electrostatic field. The 1D model has long been used to study RF discharges while retaining the self-consistent electron-energy dynamics \cite{GogolidesSawin1992}. Moreover, simplified capacitive-discharge configurations provide controlled settings of the benchmark in which numerical accuracy and efficiency can be assessed without the additional complexity associated with multidimensional geometries \cite{Turner2013}. These features make the 1D-CCP a suitable setting for examining the numerical behavior of time-advancement, flux-discretization and field-charge coupling schemes before extending them to multidimensional configurations. The model considered here consists of four governing equations of three different types and is introduced below.

\begin{figure}
\centering
\includegraphics[width=0.65\textwidth]{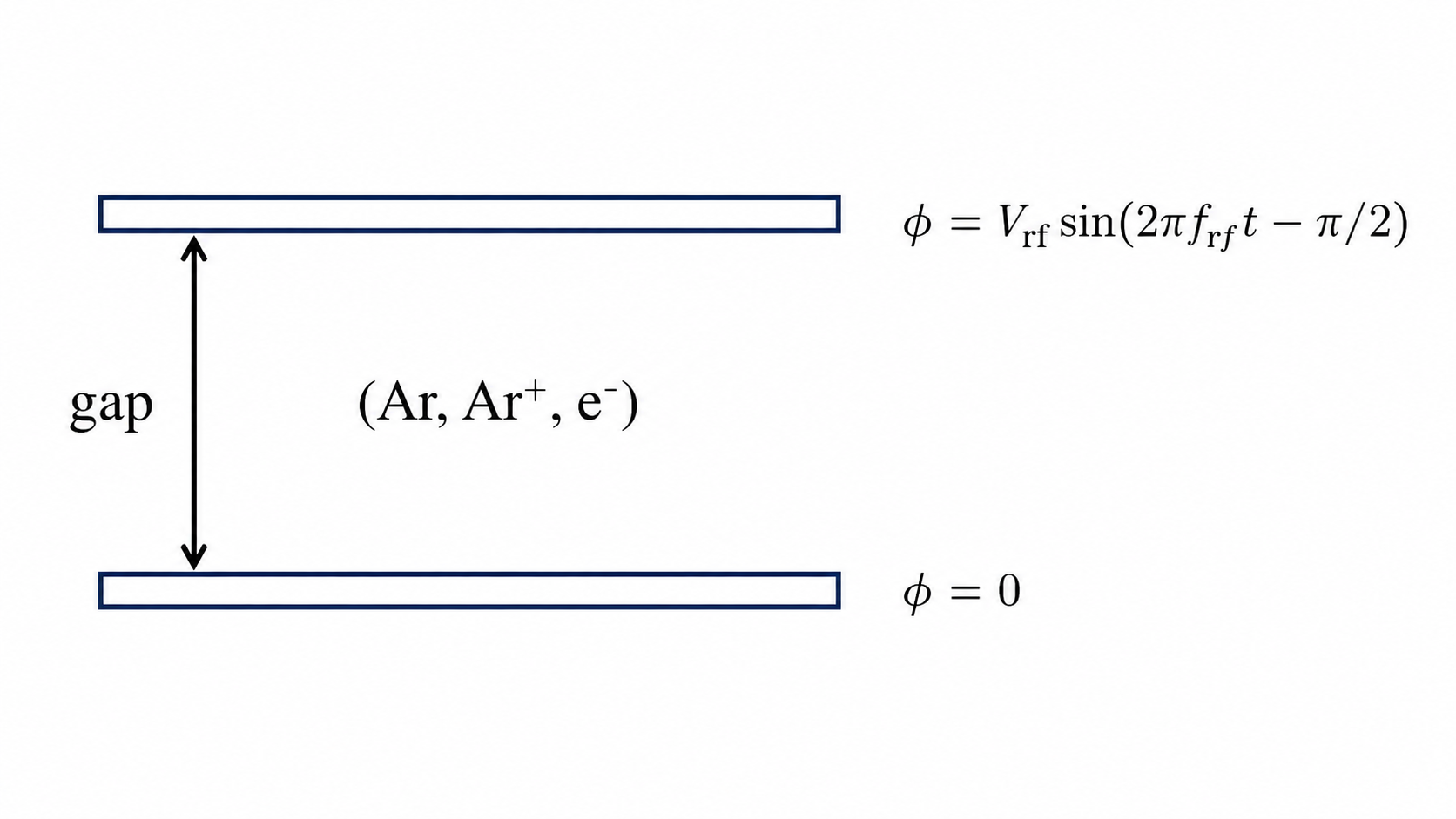} % 可以根據需要調整圖片寬度比例
\caption{1D-CCP in the domain $[0,L]$}
\label{fig:bound}
\end{figure}

\begin{subequations}\label{eq:system}
\begin{align}
    \frac{\partial^2 \phi}{\partial x^2}
    = \frac{e}{\epsilon_0}(n_e-n_+),
    \label{eq:poisson}\\
    \frac{\partial n_e}{\partial t}
    + \frac{\partial \mbox{J}_e}{\partial x}
    = S_e,
    \label{eq:1}\\
    \frac{\partial n_+}{\partial t}
    + \frac{\partial \mbox{J}_+}{\partial x}
    = S_+,
    \label{eq:2}\\
\frac{\partial E_{e}}{\partial t}+\frac{\partial \mbox{q}_{e}}{\partial x}-e\mbox{J}_{e} \cdot \frac{\partial \phi}{\partial x}+H_{r}S_{e}=0. \label{eq:3a}
\end{align}
\end{subequations}

In \eqref{eq:system}, $n_{e}$ and $n_{+}$ are the densities of the electron and ion. $\phi$ is the electric potential which is computed by solving the Poisson equation \eqref{eq:poisson} with $n_{e}$ and $n_{+}$. $e$ and $\epsilon_{0}$ are the elementary charge and vacuum permittivity. \eqref{eq:1} and \eqref{eq:2} are the continuity equations governing the evolution of $n_{e}$ and $n_{+}$. The corresponding fluxes, $\mbox{J}_{e}$ and $\mbox{J}_{+}$, are in the form of drift-diffusion terms as
\begin{equation}
\mbox{J}_{e} = -D_{e}\frac{\partial n_{e}}{\partial x}+\mu_{e}n_{e}\frac{\partial\phi}{\partial x}, \quad
\mbox{J}_{+} = -D_{+}\frac{\partial n_{+}}{\partial x}-\mu_{+}n_{+}\frac{\partial\phi}{\partial x}. \label{eq:flux}
\end{equation}
In \eqref{eq:flux}, $D_{l}, \mu_{l}>0$ denote the diffusion coefficients and drift mobilities respectively with $l=\{e,+\}$. Under typical conditions, we have $D_{e} \gg D_{+}$ and $\mu_{e} \gg \mu_{+}$. \eqref{eq:3a} is the electron energy equation with 
\begin{equation}
E_{e} = \frac{3}{2}n_{e}kT_{e}, \quad
\mbox{q}_{e} = -\frac{3}{2}kD_{e}n_{e}\frac{\partial T_{e}}{\partial x}+\frac{5}{2}kT_{e}\mbox{J}_{e}. \label{eq:fluxee}
\end{equation}
In \eqref{eq:fluxee}, $E_{e}$ and $\mbox{q}_{e}$ are the electron energy and electron energy flux. $T_{e}$ denotes the electron temperature. $k>0$ and $H_{r}>0$ are the Boltzmann constant and energy lost per chemical reaction. For the chemical reaction term $S_{e}$ and $S_{+}$ in \eqref{eq:1}-\eqref{eq:3a}, we have
\begin{equation}
S_{e} = S_{+} = k_{r}N_{0}n_{e}. \label{eq:source}
\end{equation}
In \eqref{eq:source}, $k_{r}>0$ is the reaction coefficient which is influenced by $T_{e}$ and updated by solving \eqref{eq:3a}. $N_{0}>0$ is the density of the neutral gas within the domain.

Low-temperature plasma fluid models are strongly coupled and inherently multiscale. Electron transport and energy relaxation occur much faster than ion and neutral transport, while sheath expansion and collapse, electrostatic-field evolution, chemical reactions and processes on the surface introduce further time scales \cite{AlvesMarques2012, Kushner2009}. In the RF-driven CCP, resolving these processes over many cycles to reach a periodic steady state makes time integration a major computational expense.

A common strategy for improving the efficiency of the plasma-fluid simulation is to enlarge the time step through the implicit or the semi-implicit treatment of the stiff terms, including the semi-implicit field-charge coupling, implicit electron-energy sources, improved explicit drift-diffusion schemes and the fully coupled implicit solvers \cite{ArslanbekovKolobov2021, HagelaarKroesen2000, Teunissen2020, Ventzek1993, VermaVenkattraman2021}. Although these approaches relax the constraints for the numerical stability, the stronger implicit coupling also increases the algebraic complexity and the cost per update. Thus, a larger stable time step does not necessarily reduce the overall computational time.

The cost of the long-time RF plasma simulations also depends on the number of cycles needed to reach a periodic state. Accordingly, convergence acceleration, period-based parallelization, and multiple-time-step strategies have been explored to reduce the integration  cost \cite{Kim2004, Kwon2017, LiWang2025}. The efficiency of the whole simulation reflects a trade-off among the allowable time-step range, per-step cost, numerical robustness and the long-time accuracy. 

\subsection{Main contribution}
This work systematically studies the numerical schemes within a common finite-volume 1D-CCP framework. We propose a new mixed numerical strategy combining a semi-implicit solver for the Poisson equation with the implicit Scharfetter--Gummel discretizations for the continuity and the electron-energy equations \cite{SG2,SG1}. Some alternative schemes are also considered and evaluated in terms of the maximum usable time step, total simulation time, accuracy relative to a baseline solution and the numerical robustness. Although examined in the 1D cases, these methods extend naturally to the simulation of 2D-CCP. To the best of our knowledge, this work provides the first mathematical numerical analysis of the time-advancement schemes for the simulation of CCP that balances the theoretical considerations with the practical performance.

\subsection{Outline and notation}
The remainder of this paper is organized as follows. \cref{sec:schemes} introduces the existing schemes and our proposed SG-based scheme for the 1D-CCP. We provide the corresponding numerical analysis of our scheme in \cref{sec:analysis}. Numerical experiments are presented in \cref{sec:experiments}, followed by conclusions in \cref{sec:conclusions}.

We use the following notation in this work. For $X \in \mathbb{R}^{n\times n}$ and $b \in \mathbb{R}^n$, $X,b \ge 0$ means that all their entries are nonnegative. For $X \in \mathbb{R}^{n\times n}$, $\sigma_n(X)$ and $\lambda_n(X)$ denote its $n$-th largest singular value and eigenvalue, respectively. For the norms of $A$ and $b$, we take $\norm{b}_1=\sum\limits_{i=0}^{n-1}\abs{b_i}, \quad \norm{b}_2=\sqrt{\sum\limits_{i=0}^{n-1}b_i^2}, \quad \norm{A}_1=\max \limits_{0\leq j\leq n-1}\sum \limits_{i=0}^{n-1}\abs{a_{ij}}, \quad \norm{A}_2=\sigma_1(A)$.
For $A,B \in \mathbb{R}^{m\times n}$, $A \ge B$ denotes the entrywise inequality, $a_{ij} \ge b_{ij}$, for all $i$ and $j$.

\section{Numerical schemes of the model}
\label{sec:schemes}
This section presents numerical schemes for solving the system of the 1D-CCP \eqref{eq:system}. We discretize the spatial domain $[0,L]$ using a uniform mesh with spacing $\Delta x=\frac{L}{N}$, where $N$ is the number of cells. The cells are indexed from $0$ to $N-1$ with the interface between the $i$-th and the $i+1$-th cells denoted by $i+\frac{1}{2}$, see \cref{fig:grid}. Let $c_i$ denote the value of $c \in \{n_e, n_+, T_e, \phi, S_e\}$ in the $i$-th cell. Using
$-e\mbox{J}_e\frac{\partial\phi}{\partial x}
=-e\frac{\partial(\mbox{J}_e\phi)}{\partial x}
+e\phi\frac{\partial\mbox{J}_e}{\partial x}$, \eqref{eq:3} can be rewritten as
\begin{align}
\frac{\partial E_{e}}{\partial t}+\frac{\partial \mbox{G}}{\partial x} &+e\phi \cdot \frac{\partial \mbox{J}_{e}}{\partial x}=-H_{r}S_{e}, \label{eq:3} \\
\mbox{G} &=\mbox{q}_{e}-e \cdot (\mbox{J}_{e}\phi). \label{eq:fluxg}
\end{align}
We use \eqref{eq:3}-\eqref{eq:fluxg} to update $T_e$, which is equivalent to solving \eqref{eq:3a}. For a quantity $h \in \{n_e, n_+, T_e, \phi, S_e, \mbox{J}_e, \mbox{J}_+, \mbox{G}\}$, $h^n$ denotes its value at the $n$-th time step with $\Delta t$ the time step in the simulation.

\begin{figure}
\centering
\includegraphics[width=0.60\textwidth]{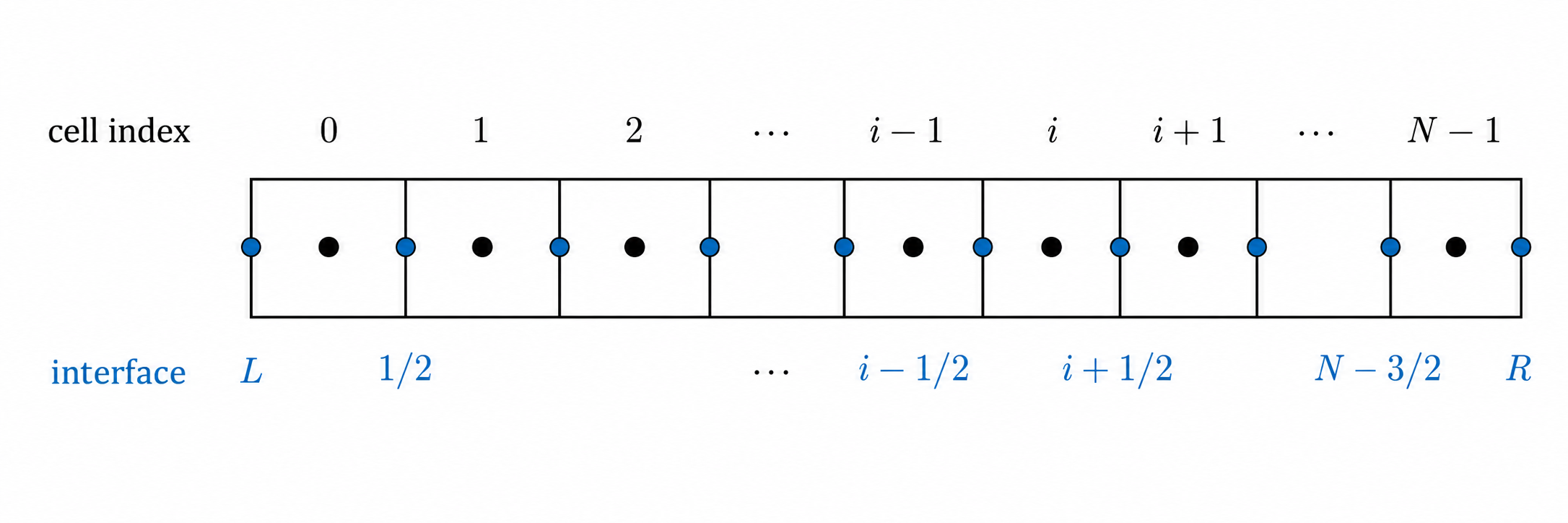} % 可以根據需要調整圖片寬度比例
\caption{Cells and interfaces in the domain $[0,L]$}
\label{fig:grid}
\end{figure}

\subsection{Existing numerical schemes}
Several numerical schemes have been developed for simulating the 1D-CCP. For the Poisson’s equation \eqref{eq:poisson}, an explicit finite difference discretization can be applied at the $i$-th interior grid points with $1 \le i \le N-2$, which is in the form of 
\begin{equation}
\frac{\partial^{2}\phi^{n+1}}{\partial x^{2}}=\frac{\phi_{i+1}^{n+1}-2\phi_{i}^{n+1}+\phi_{i-1}^{n+1}}{(\Delta x)^{2}} =\frac{e}{\epsilon_{0}} \cdot (n_{e,i}^{n}-n_{+,i}^{n}), \label{eq:poissone} 
\end{equation}
to compute $\phi^{n+1}$ at the $(n+1)$-th time step. For the continuity equations and the electron energy equation \eqref{eq:1}-\eqref{eq:3a}, they can be reformulated in a unified form using
$E_{e,i}^{n+1}=\frac{3}{2}n_{e,i}^{n+1}kT_{e,i}^{n+1}$ as
\begin{align}
\frac{W_{i}^{n + 1} - W_{i}^{n}}{\Delta t} &+ \frac{\mbox{F}_{i + \frac{1}{2}}^{p} - \mbox{F}_{i - \frac{1}{2}}^{p}}{\Delta x} +S_{i}^{n} = 0, \quad  
W_{i}^{n+1} = \begin{bmatrix} n_{e,i}^{n+1} \\ n_{+,i}^{n+1} \\ T_{e,i}^{n+1} \end{bmatrix} , \quad W_{i}^{n} = \begin{bmatrix} n_{e,i}^{n} \\ n_{+,i}^{n} \\ \frac{2E_{e,i}^{n}}{3kn_{e,i}^{n+1}} \end{bmatrix}, \label{eq:tuie} \\ 
S_{i}^{n} &= \begin{bmatrix} -S_{e,i}^{n} \\ -S_{+,i}^{n} \\ \frac{2}{3kn_{e,i}^{n+1}} \cdot H_{r}S_{e,i}^{n} \end{bmatrix}, \quad
\mbox{F}_{i + \frac{1}{2}}^{p} = \begin{bmatrix} \mbox{J}_{e,i + \frac{1}{2}}^{p} \\ \mbox{J}_{+,i + \frac{1}{2}}^{p} \\ \frac{2}{3kn_{e,i}^{n+1}} \cdot \left( \mbox{G}_{i + \frac{1}{2}}^{p} + e\phi_{i}^{n+1}\mbox{J}_{e,i + \frac{1}{2}}^{n} \right) \end{bmatrix}, \label{eq:definition} \\
\mbox{J}_{e,i + \frac{1}{2}}^{p} &= - D_{e} \cdot \frac{n_{e,i + 1}^{p} - n_{e,i}^{p}}{\Delta x}+\mu_{e} \cdot n_{e,i + \frac{1}{2}}^{\mathrm{u},p} \cdot \frac{\phi_{i + 1}^{n+1} - \phi_{i}^{n+1}}{\Delta x}, \label{eq:je12} \\
\mbox{J}_{+,i + \frac{1}{2}}^{p} &= - D_{+} \cdot \frac{n_{+,i + 1}^{p} - n_{+,i}^{p}}{\Delta x}- \mu_{+}n_{+,i + \frac{1}{2}}^{\mathrm{d},p} \cdot \frac{\phi_{i + 1}^{n+1} - \phi_{i}^{n+1}}{\Delta x}, \label{eq:j+12} \\
\mbox{G}_{i + \frac{1}{2}}^{p} &= - \frac{3}{2}kD_{e}n_{e,i + \frac{1}{2}}^{\mathrm{u},n} \cdot \frac{T_{e,i + 1}^{p} - T_{e,i}^{p}}{\Delta x}+\left( \frac{5}{2}kT_{e,i + \frac{1}{2}}^{\mathrm{u},p} - e\phi_{i + \frac{1}{2}}^{\mathrm{u},n+1} \right) \cdot \mbox{J}_{e,i + \frac{1}{2}}^{n}. \label{eq:g12} 
\end{align}
The source term in \eqref{eq:tuie} is treated explicitly as $S_{e,i}^{n}=k_{r,i}^{n}N_{0}n_{e,i}^{n}$, where $k_{r,i}^{n}$ is computed using $T_{e,i}^{n-1}$. For the variables $m_{i+\frac{1}{2}}^{u}$ with $m \in \{n_e,\phi,T_e\}$ on the interface and $n_{+,i+\frac{1}{2}}^{d}$ in \eqref{eq:je12}-\eqref{eq:g12}, the upwind and the downwind schemes are defined based on $\phi$ as
\begin{align} 
m_{i + \frac{1}{2}}^{\mathrm{u}} = \left\{
\begin{array}{ll}
m_{i}, & \phi_{i + 1} \ge \phi_{i} \\
m_{i+1}, & \phi_{i + 1}<\phi_{i}
\end{array}, 
\right. \quad 
n_{+,i + \frac{1}{2}}^{\mathrm{d}} = \left\{
\begin{array}{ll}
n_{+,i+1}, & \phi_{i+1} \ge \phi_{i} \\
n_{+,i}, & \phi_{i+1}<\phi_{i}
\end{array}.  \label{eq:wind}
\right.
\end{align}
For the explicit discretization of \eqref{eq:tuie}, we take $p=n$, whereas the implicit discretization is done with $p=n+1$. To update $T_e^{n+1}$ explicitly, $E_e^{n+1}$ is first computed and $T_e^{n+1}$ is then computed using $n_e^{n+1}$ obtained from the electron continuity equation. For the implicit scheme, there are two potential approaches for $n_e^{n+1}$. The first is to solve three separate linear systems of the size $N$ to obtain $n_e^{n+1}$, $n_+^{n+1}$ and $T_e^{n+1}$ sequentially, allowing the updated $n_e^{n+1}$ to be used directly. Alternatively, the three variables can be solved simultaneously through a coupled linear system of the size $3N$. In the latter approach, $n_{e,i}^{n+1}$ is approximated by an explicit update with a sufficiently small time step $\Delta t_0$. In terms of the boundary conditions for $T_e$ and $\phi$, they are given by
\begin{align} 
T_{e,j} = \left\{
\begin{array}{ll}
0.5, & j=L \\
0.5, & j=R
\end{array}, 
\right. \quad
\phi_{j} = \left\{
\begin{array}{ll}
\phi_{RF}\mbox{sin}(2\pi f_{RF}t-\frac{\pi}{2}),  &j=L \\
0, &j=R
\end{array}.
\right. \label{eq:bb}
\end{align}
In \eqref{eq:bb}, $\phi_{RF}>0$ denotes the amplitude of the RF driving voltage and $f_{RF}>0$ is the corresponding radio frequency. For the numerical fluxes, with $p \in \{n,n+1\}$, we have
\begin{align}
\mbox{J}_{+,L}^{p} &= \left\{
\begin{array}{ll}
-2\mu_{+}n_{+,0}^{p} \cdot \frac{\phi_{0}^{n+1} - \phi_{L}^{n+1}}{\Delta x}, & \phi_{0}^{n+1} \ge \phi_{L}^{n+1} \\
0, & \phi_{0}^{n+1}<\phi_{L}^{n+1}
\end{array},
\right. \quad 
\mbox{J}_{e,L}^{p} = - K_{s}n_{e,0}^{p} - \gamma\mbox{J}_{+,L}^{n}, \label{eq:jl} \\
\mbox{G}_{L}^{p} &= \mbox{G}_{-\frac{1}{2}}^{p} = -3kD_{e}n_{e,0}^{n}\frac{T_{e,0}^{p} - T_{e,L}^{p}}{\Delta x}+\left( \frac{5}{2}k \cdot \frac{T_{e,L}^{p} + T_{e,0}^{p}}{2} - e\phi_{L}^{n+1} \right)\mbox{J}_{e,L}^{n}, \label{eq:gl} \\
\mbox{J}_{+,R}^{p} &= \left\{
\begin{array}{ll}
-2\mu_{+}n_{+,N-1}^{p} \cdot \frac{\phi_{R}^{n+1} - \phi_{N-1}^{n+1}}{\Delta x}, & \phi_{R}^{n+1} \le \phi_{N-1}^{n+1} \\
0, & \phi_{R}^{n+1}>\phi_{N-1}^{n+1}
\end{array},
\right. \quad
\mbox{J}_{e,R}^{p} = K_{s}n_{e,N-1}^{p} - \gamma\mbox{J}_{+,R}^{n}, \label{eq:jr} \\
\mbox{G}_{R}^{p} &= - 3kD_{e}n_{e,m-1}^{n}\frac{T_{e,R}^{p} - T_{e,m-1}^{p}}{\Delta x}+\left( \frac{5}{2}k \cdot \frac{T_{e,R}^{p} + T_{e,m-1}^{p}}{2} - e\phi_{R}^{n+1} \right)\mbox{J}_{e,R}^{n}. \label{eq:gr}
\end{align}
$K_{s}, \gamma>0$ in \eqref{eq:gl} and \eqref{eq:gr} are the given electron loss coefficient and a secondary-electron emission coefficient, respectively. 

\subsection{Fluxes with the Scharfetter-Gummel (SG) scheme}
In this work, we employ the Scharfetter-Gummel (SG) scheme to reconstruct the fluxes, replacing the original convection-diffusion and the upwind terms in \eqref{eq:je12}-\eqref{eq:g12}, \eqref{eq:jl} and \eqref{eq:jr}. We primarily focus on the implicit formulation, which permits larger time steps and improves the computational efficiency. For the $i+\frac{1}{2}$-th interior interfaces with $0 \le i \le N-2$, the fluxes $\mbox{J}_e$, $\mbox{J}_+$ and $\mbox{G}$ in \eqref{eq:1}-\eqref{eq:3a} at the $(n+1)$-th time step are approximated as
\begin{align}
\mbox{J}_{e,i+\frac{1}{2}}^{n+1} &=\frac{D_{e}}{\Delta x} \cdot [B(-P_{e,i+\frac{1}{2}}^{n+1})n_{e,i}^{n+1}-B(P_{e,i+\frac{1}{2}}^{n+1})n_{e,i+1}^{n+1}], \label{eq:sgje12} \\
\mbox{J}_{+,i+\frac{1}{2}}^{n+1} &=\frac{D_{+}}{\Delta x} \cdot [B(-P_{+,i+\frac{1}{2}}^{n+1})n_{+,i}^{n+1}-B(P_{+,i+\frac{1}{2}}^{n+1})n_{+,i+1}^{n+1}], \label{eq:sgj+12} \\
\mbox{G}_{i+\frac{1}{2}}^{n+1} &=\frac{3}{2}kn_{e,i+\frac{1}{2}}^{u,n} \cdot \frac{D_{e}}{\Delta x} \cdot
\left[B\!\left(-P_{T,i+\frac{1}{2}}^{n}\right) T_{e,i}^{n+1}-B\!\left(P_{T,i+\frac12}^{n}\right) T_{e,i+1}^{n+1}\right]-e\phi_{i+\frac12}^{u,n+1}\,
\mbox{J}_{e,i+\frac12}^{n}, \label{eq:rg12}
\end{align}
Here, $B(x)$ in \eqref{eq:sgje12} and \eqref{eq:sgj+12} is the Bernoulli function with 
\begin{equation}
B(x)=\frac{x}{e^{x}-1}, \quad B(x)>0, \quad B(0)=1, \quad B'(0)=-\frac{1}{2}. \label{eq:bf}
\end{equation}
For $P_{e}$, $P_{+}$ and $P_{T}$ in \eqref{eq:sgje12}-\eqref{eq:rg12}, we have
\begin{equation}
P_{e,i+\frac{1}{2}} =\mu_{e} \cdot \frac{\phi_{i+1}-\phi_{i}}{D_{e}}, \quad P_{+,i+\frac{1}{2}} =-\mu_{+} \cdot \frac{\phi_{i+1}-\phi_{i}}{D_{+}}, \quad  P_{T,i+\frac12}=
\frac{5\mbox{J}_{e,i+\frac12} \cdot \Delta x}{3D_{e}n_{e,i+\frac12}^{u}}. \label{eq:p}
\end{equation}
For the fluxes of $n_e$ and $n_+$ on the boundary, we retain the discretizations in \eqref{eq:jl} and \eqref{eq:jr}. The fluxes of the electron energy on the boundary, $\mbox{G}_L$ and $\mbox{G}_R$, originally defined in \eqref{eq:gl} and \eqref{eq:gr}, are modified as
\begin{align}
\mbox{G}_{L}^{n+1} &=\frac{2\kappa_{e,0}^{n}}{\Delta x} \cdot \left[B\!\left(-P_{e,L}^{n}\right) T_{e,L}^{n+1}-B\!\left(P_{e,L}^{n}\right)T_{e,0}^{n+1}\right]-e\,\phi_L^{n+1}\,\mbox{J}_{e,L}^{n}, \label{eq:glr} \\
\mbox{G}_{R}^{n+1} &= \frac{2\kappa_{e,N-1}^{n}}{\Delta x} \cdot \left[B\!\left(-P_{e,R}^{n}\right) T_{e,N-1}^{n+1}-B\!\left(P_{e,R}^{n}\right) T_{e,R}^{n+1}\right]-e\,\phi_R^{n+1}\,\mbox{J}_{e,R}^{n}, \label{eq:grr}
\end{align}
For $\kappa_{e,i}$, $P_{e,L}$ and $P_{e,R}$ in \eqref{eq:glr} and \eqref{eq:grr} with $i=\{0,N-1\}$, we define
\begin{equation}
\kappa_{e,i}=\frac{3}{2}kD_{e}n_{e,i}, \quad P_{e,L}=\frac{5\mbox{J}_{e,L} \cdot \Delta x}{6D_{e}n_{e,0}}, \quad P_{e,R}=\frac{5\mbox{J}_{e,R} \cdot \Delta x}{6D_{e}n_{e,N-1}}. \label{eq:setb}
\end{equation}

\subsection{Solving the Poisson's equation with a semi-implicit scheme}
In this work, we take a SG-based semi-implicit scheme for the Poisson's equation. To compute $\phi^{n+1}$ from \eqref{eq:poisson}, the left-hand side (LHS) is approximated using a second-order central difference scheme and both sides are multiplied by $\epsilon_0$. The resulting discrete formulation is given by
\begin{equation}
\frac{\epsilon_{0} \cdot (\phi_{i+1}^{n+1}-2\phi_{i}^{n+1}+\phi_{i-1}^{n+1})}{(\Delta x)^{2}} = e \cdot (n_{e,i}^{n+1}-n_{+,i}^{n}), \label{eq:poissoni}
\end{equation}
with the right-hand side (RHS) of \eqref{eq:poissoni} treated semi-implicitly when $1 \le i \le N-2$.
The semi-implicit treatment accounts for the much higher mobility of the electrons compared with the ions. Solving \eqref{eq:poissoni} for $\phi^{n+1}$ then reduces to a linear system
\begin{equation}
\textbf{A}^{n}\phi^{n+1}=\textbf{b}^{n}. \label{eq:ax=b} 
\end{equation}
In the following, we formulate the semi-implicit SG-based scheme for the Poisson's equation using the fluxes and boundary conditions defined before, yielding $\mathbf{A}^{n}\in\mathbb{R}^{N\times N}$ and $\mathbf{b}^{n}\in\mathbb{R}^{N}$ in \eqref{eq:ax=b}.

\subsubsection{Interior cells}
For the interior cells of \eqref{eq:poissoni} with $1 \le i \le N-2$, we can calculate $\frac{\partial \mbox{J}_{e,i+\frac{1}{2}}}{\partial (\phi_{i+1}-\phi_{i})}$ with \eqref{eq:sgje12} as
\begin{align}
\frac{\partial \mbox{J}_{e,i+\frac{1}{2}}}{\partial (\phi_{i+1}-\phi_{i})} &=-\frac{\mu_{e}}{\Delta x} \cdot [n_{e,i}B'(-P_{e,i+\frac{1}{2}})+n_{e,i+1}B'(P_{e,i+\frac{1}{2}})] =\frac{A_{e,i+\frac{1}{2}}}{\Delta x}, \label{eq:conduct} \\
A_{e,i+\frac{1}{2}} &=-\mu_{e} \cdot [n_{e,i}B'(-P_{e,i+\frac{1}{2}})+n_{e,i+1}B'(P_{e,i+\frac{1}{2}}),. \label{eq:aei12}
\end{align}
with $P_{e,i+\frac{1}{2}}$ defined in \eqref{eq:p}. When $P_{e,i+\frac{1}{2}} \to 0$, with $B'(0)=-\frac{1}{2}$ in \eqref{eq:bf}, we can get
\begin{equation}
A_{e,i+\frac{1}{2}}=\mu_{e} \cdot \frac{n_{e,i}+n_{e,i+1}}{2}. \label{eq:limit}
\end{equation}

For the semi-implicit scheme of the Poisson's equation to compute $\phi^{n+1}$, with the Taylor's formula and \eqref{eq:conduct}, $\mbox{J}_{e,i+\frac{1}{2}}^{n+1}$ can be approximated with $\mbox{J}_{e,i+\frac{1}{2}}^{n}$ as
\begin{align} 
\begin{split} \label{eq:approxj}
\mbox{J}_{e,i+\frac{1}{2}}^{n+1} &=\mbox{J}_{e,i+\frac{1}{2}}^{n}+\frac{\partial \mbox{J}_{e,i+\frac{1}{2}}^{n}}{\partial (\phi_{i+1}^{n}-\phi_{i}^{n})} \cdot [(\phi_{i+1}^{n+1}-\phi_{i}^{n+1})-(\phi_{i+1}^{n}-\phi_{i}^{n})] \\ &=\mbox{J}_{e,i+\frac{1}{2}}^{n} +\frac{A_{e,i+\frac{1}{2}}^{n}}{\Delta x} \cdot [(\phi_{i+1}^{n+1}-\phi_{i}^{n+1})-(\phi_{i+1}^{n}-\phi_{i}^{n})] \\ &=\mbox{r}_{e,i+\frac{1}{2}}^{n}+\frac{A_{e,i+\frac{1}{2}}^{n}}{\Delta x} \cdot (\phi_{i+1}^{n+1}-\phi_{i}^{n+1}), 
\end{split} \\
\mbox{r}_{e,i+\frac{1}{2}}^{n} &=\mbox{J}_{e,i+\frac{1}{2}}^{n}-\frac{A_{e,i+\frac{1}{2}}^{n}}{\Delta x} \cdot (\phi_{i+1}^{n}-\phi_{i}^{n}). \label{eq:residual}
\end{align}
In \eqref{eq:tuie}, when $p=n+1$, we can compute $n_{e,i}^{n+1}$ with $\mbox{J}_{e,i+\frac{1}{2}}^{n+1}$ and $\mbox{J}_{e,i-\frac{1}{2}}^{n+1}$. Therefore, with \eqref{eq:tuie}, \eqref{eq:approxj} and \eqref{eq:residual}, we can approximate the RHS of \eqref{eq:poissoni} as
\begin{equation} \label{eq:simplifysg}
\begin{split}
n_{e,i}^{n+1}-n_{+,i}^{n+1} &= n_{e,i}^{n}-n_{+,i}^{n}-\frac{\Delta t}{\Delta x} \cdot (\mbox{J}_{e,i+\frac{1}{2}}^{n+1}-\mbox{J}_{e,i-\frac{1}{2}}^{n+1})+\Delta tS_{e,i}^{n} \\ &= n_{e,i}^{n}-n_{+,i}^{n}-\frac{\Delta t}{\Delta x} \cdot (\mbox{r}_{e,i+\frac{1}{2}}^{n}+\frac{A_{e,i+\frac{1}{2}}^{n}}{\Delta x} \cdot (\phi_{i+1}^{n+1}-\phi_{i}^{n+1})) \\ &+\frac{\Delta t}{\Delta x} \cdot (\mbox{r}_{e,i-\frac{1}{2}}^{n}+\frac{A_{e,i+\frac{1}{2}}^{n}}{\Delta x} \cdot (\phi_{i}^{n+1}-\phi_{i-1}^{n+1}))+\Delta tS_{e,i}^{n}. 
\end{split}
\end{equation}
We put \eqref{eq:simplifysg} into \eqref{eq:poissoni} and after simplification, it follows that 
\begin{equation} \label{eq:constructsg}
\begin{split}
\frac{\epsilon_{0}+e\Delta t\cdot A_{e,i-\frac{1}{2}}^{n}}{(\Delta x)^{2}} \cdot \phi_{i-1}^{n+1} &-\frac{2\epsilon_{0}+e\Delta t\cdot (A_{e,i-\frac{1}{2}}^{n}+A_{e,i+\frac{1}{2}}^{n})}{(\Delta x)^{2}} \cdot \phi_{i}^{n+1}+\frac{\epsilon_{0}+e\Delta t \cdot A_{e,i+\frac{1}{2}}^{n}}{(\Delta x)^{2}} \cdot \phi_{i+1}^{n+1} \\ &=e(n_{e,i}^{n} - n_{+,i}^{n})-e\cdot \frac{\Delta t}{\Delta x} \cdot (\mbox{r}_{e,i+\frac{1}{2}}^{n}-\mbox{r}_{e,i-\frac{1}{2}}^{n})+e\Delta tS_{e,i}^{n}. 
\end{split}
\end{equation}
Therefore, based on \eqref{eq:constructsg}, for the $l$-th row of $\textbf{A}^{n}$ and $\textbf{b}^{n}$ in \eqref{eq:ax=b} with $2 \le l \le N-1$, we can get
\begin{equation} \label{eq:iitemsg}
\begin{split}
a_{l-1,l}^{n} &=\frac{\epsilon_{0}+e\Delta t\cdot A_{e,l-\frac{3}{2}}^{n}}{(\Delta x)^{2}}, \quad a_{l,l}^{n}=-\frac{2\epsilon_{0}+e\Delta t\cdot (A_{e,l-\frac{3}{2}}^{n}+A_{e,l-\frac{1}{2}}^{n})}{(\Delta x)^{2}}, \\ a_{l,l+1}^{n} &=\frac{\epsilon_{0}+e\Delta t\cdot A_{e,l-\frac{1}{2}}^{n}}{(\Delta x)^{2}}, \quad b_{l}^{n}=e(n_{e,l-1}^{n} - n_{+,l-1}^{n})-e\cdot \frac{\Delta t}{\Delta x} \cdot (\mbox{r}_{e,l-\frac{1}{2}}^{n}-\mbox{r}_{e,l-\frac{3}{2}}^{n})+e\Delta tS_{e,i}^{n}.
\end{split}
\end{equation}

\subsubsection{Boundary-adjacent cells} 
Having treated the interior cells, we now consider the boundary-adjacent cells with $i=0$ and $i=N-1$.

\textbf{i=0} 

When $i=0$, similar to \eqref{eq:simplifysg}, with \eqref{eq:tuie}, \eqref{eq:approxj} and \eqref{eq:residual}, we can have
\begin{equation}
\begin{split}
n_{e,0}^{n + 1}-n_{+,0}^{n + 1} &= n_{e,0}^{n}-n_{+,0}^{n}-\frac{\Delta t}{\Delta x} \cdot (\mbox{J}_{e,\frac{1}{2}}^{n+1}-\mbox{J}_{e,L}^{n})+\Delta tS_{e,0}^{n} \\ &= n_{e,0}^{n}-n_{+,0}^{n}-\frac{\Delta t}{\Delta x} \cdot (\mbox{r}_{e,\frac{1}{2}}^{n}+\frac{A_{e,\frac{1}{2}}^{n}}{\Delta x} \cdot (\phi_{1}^{n+1}-\phi_{0}^{n+1})-\mbox{J}_{e,L}^{n})+\Delta tS_{e,0}^{n}. \label{eq:simplifysg1}
\end{split}
\end{equation}
Here, $\mbox{J}_{e,L}$ is defined in \eqref{eq:jl}. With \eqref{eq:simplifysg1}, we discretize the LHS of \eqref{eq:poisson} at the left boundary using $\phi_L$ into
\begin{equation} \label{eq:expand2}
\frac{\frac{\phi_{1}^{n+1}-\phi_{0}^{n+1}}{\Delta x}-\frac{\phi_{0}^{n+1}-\phi_{L}^{n+1}}{\frac{\Delta x}{2}}}{\Delta x} = e \cdot [n_{e,0}^{n}-n_{+,0}^{n}-\frac{\Delta t}{\Delta x} \cdot (\mbox{r}_{e,\frac{1}{2}}^{n}+\frac{A_{e,\frac{1}{2}}^{n}}{\Delta x} \cdot (\phi_{1}^{n+1}-\phi_{0}^{n+1})-\mbox{J}_{e,L}^{n})+\Delta tS_{e,0}^{n}].
\end{equation}
$\phi_{L}$ in \eqref{eq:expand2} is defined in \eqref{eq:bb}. We simplify \eqref{eq:expand2} into
\begin{equation} \label{eq:ax=b2a}
\frac{-3\epsilon_{0}-e\Delta t \cdot A_{e,\frac{1}{2}}^{n}}{(\Delta x)^{2}} \cdot \phi_{0}^{n+1}+\frac{\epsilon_{0}+e\Delta t \cdot A_{e,\frac{1}{2}}^{n}}{(\Delta x)^{2}} \cdot \phi_{1}^{n+1}=b_{1}^{n}.
\end{equation}
Therefore, by applying \eqref{eq:ax=b2a}, the first row of $\mathbf{A}^{n}$ and the corresponding entry of $\mathbf{b}^{n}$ in \eqref{eq:ax=b} are given by
\begin{equation} \label{eq:ax=b1a}
\begin{split}
a_{1,1}^{n} &=\frac{-3\epsilon_{0}-e\Delta t \cdot A_{e,\frac{1}{2}}^{n}}{(\Delta x)^{2}}, \quad a_{1,2}^{n}=\frac{\epsilon_{0}+e\Delta t \cdot A_{e,\frac{1}{2}}^{n}}{(\Delta x)^{2}}, \\ b_{1}^{n} &=e \cdot [n_{e,0}^{n}-n_{+,0}^{n}-\frac{\Delta t}{\Delta x} \cdot (\mbox{r}_{e,\frac{1}{2}}^{n}-\mbox{J}_{e,L}^{n})+\Delta tS_{e,0}^{n}]-\frac{2\epsilon_{0}\phi_{L}^{n}}{(\Delta x)^{2}}.
\end{split}
\end{equation}

\textbf{i=N-1} 

When $i=N-1$, similar to \eqref{eq:simplifysg1}, with \eqref{eq:tuie}, \eqref{eq:approxj} and \eqref{eq:residual}, we can have
\begin{equation}
\begin{split}
&n_{e,N-1}^{n + 1} -n_{+,N-1}^{n + 1} \\ 
= \, &n_{e,N-1}^{n} -n_{+,N-1}^{n}-\frac{\Delta t}{\Delta x} \cdot (\mbox{J}_{e,R}^{n}-\mbox{J}_{e,N-\frac{3}{2}}^{n})+\Delta tS_{e,N-1}^{n} \\ 
= \, &\underbrace{n_{e,N-1}^{n} -n_{+,N-1}^{n} -\frac{\Delta t}{\Delta x} \cdot (\mbox{J}_{e,R}^{n}-\mbox{r}_{e,N-\frac{3}{2}}^{n}-\frac{A_{e,N-\frac{3}{2}}^{n}}{\Delta x} \cdot (\phi_{N-1}^{n+1}-\phi_{N-2}^{n+1}))+\Delta tS_{e,N-1}^{n}}_{R_{b}}. \label{eq:simplify3}
\end{split}
\end{equation}
Here, $\mbox{J}_{e,R}$ is defined in \eqref{eq:jr}. Similar to \eqref{eq:expand2}, we discretize the LHS of \eqref{eq:poisson} with $\phi_{R}$ on the right boundary and with \eqref{eq:simplify3}, we have

Here, $\mbox{J}_{e,R}$ is defined in \eqref{eq:jr}. With \eqref{eq:simplify3}, we discretize the LHS of \eqref{eq:poisson} at the right boundary using $\phi_R$ into
\begin{equation} \label{eq:expand3}
\frac{\frac{\phi_{R}^{n}-\phi_{N-1}^{n}}{\frac{\Delta x}{2}}-\frac{\phi_{N-1}^{n}-\phi_{N-2}^{n}}{\Delta x}}{\Delta x} = e \cdot R_{b}.
\end{equation}
$\phi_{R}$ in \eqref{eq:expand3} is defined in \eqref{eq:bb}. We simplify \eqref{eq:expand3} into
\begin{equation} \label{eq:ax=b4}
\frac{-3\epsilon_{0}-e\Delta t \cdot A_{e,N-\frac{3}{2}}^{n}}{(\Delta x)^{2}} \cdot \phi_{N-1}^{n+1}+\frac{\epsilon_{0}+e\Delta t \cdot A_{e,N-\frac{3}{2}}^{n}}{(\Delta x)^{2}} \cdot \phi_{N-2}^{n+1}=b_{N}^{n}.
\end{equation}
Therefore, using \eqref{eq:ax=b4}, the $N$-th row of $\mathbf{A}^{n}$ and the corresponding entry of $\mathbf{b}^{n}$ in \eqref{eq:ax=b} are given by
\begin{equation} \label{eq:ax=b5}
\begin{split}
a_{N,N}^{n} &=\frac{-3\epsilon_{0}-e\Delta t \cdot A_{e,N-\frac{3}{2}}^{n}}{(\Delta x)^{2}}, \quad a_{N,N-1}^{n}=\frac{\epsilon_{0}+e\Delta t \cdot A_{e,N-\frac{3}{2}}^{n}}{(\Delta x)^{2}}, \\ b_{N}^{n} &= e \cdot [n_{e,N-1}^{n}-n_{+,N-1}^{n}-\frac{\Delta t}{\Delta x} \cdot (\mbox{J}_{e,R}^{n}-\mbox{r}_{e,N-\frac{3}{2}}^{n})+\Delta tS_{e,N-1}^{n}]-\frac{2\epsilon_{0}\phi_{R}^{n}}{(\Delta x)^{2}}.
\end{split}
\end{equation}

\subsubsection{Summary of $\textbf{A}^{n}\phi^{n+1}=\textbf{b}^{n}$} 
In summary, the linear system \eqref{eq:ax=b} for computing $\phi^{n+1}$ has $\mathbf{A}^{n} \in \mathbb{R}^{N\times N}$ and $\mathbf{b}^{n} \in \mathbb{R}^{N}$, with entries specified by \eqref{eq:iitemsg}, \eqref{eq:ax=b1a} and \eqref{eq:ax=b5}, where $A_e$ is defined in \eqref{eq:limit}. Moreover, $\mathbf{A}^{n}$ is a symmetric tridiagonal matrix.

%\begin{equation} \label{eq:summary3}
%a_{l,l-1}^{n} =\frac{\epsilon_{0}+e\Delta t\cdot A_{e,l-\frac{3}{2}}^{n}}{(\Delta x)^{2}}, \quad a_{l-1,l}^{n}=\frac{\epsilon_{0}+e\Delta t\cdot A_{e,l-\frac{3}{2}}^{n}}{(\Delta x)^{2}}, \quad a_{l-1, l-1}^{n} = \left\{
%\begin{array}{ll}
%\frac{-3\epsilon_{0}-e\Delta t \cdot A_{e,\frac{1}{2}}^{n}}{(\Delta x)^{2}}, & l=2 \\
%-\frac{2\epsilon_{0}+e\Delta t\cdot (A_{e,l-\frac{5}{2}}^{n}+A_{e,l-\frac{3}{2}}^{n})}{(\Delta x)^{2}}, & 3 \le l \le N \\
%\frac{-3\epsilon_{0}-e\Delta t \cdot A_{e,N-\frac{3}{2}}^{n}}{(\Delta x)^{2}}, & l=N+1
%\end{array}, 
%\right. \\
%\end{equation}
%\begin{equation} \label{eq:summary4}
%b_{l-1}^{n} = \left\{
%\begin{array}{ll}
%e \cdot [n_{e,0}^{n}-n_{+,0}^{n}-\frac{\Delta t}{\Delta x} \cdot (\mbox{r}_{e,\frac{1}{2}}^{n}-\mbox{J}_{e,L}^{n})+\Delta t S_{e,0}^{n}]-\frac{2\epsilon_{0}\phi_{L}^{n}}{(\Delta x)^{2}}, &l=2 \\
%e \cdot (n_{e,l-2}^{n} - n_{+,l-2}^{n})-e\cdot \frac{\Delta t}{\Delta x} \cdot (\mbox{r}_{e,l-\frac{3}{2}}^{n}-\mbox{r}_{e,l-\frac{5}{2}}^{n}+\Delta tS_{e,l-2}^{n}), & 3 \le l \le N \\
%e \cdot [n_{e,N-1}^{n}-n_{+,N-1}^{n}-\frac{\Delta t}{\Delta x} \cdot (\mbox{J}_{e,R}^{n}-\mbox{r}_{e,N-\frac{3}{2}}^{n})+\Delta tS_{e,N-1}^{n}]-\frac{2\epsilon_{0}\phi_{R}^{n}}{(\Delta x)^{2}}, & l=N+1
%\end{array}.
%\right.
%\end{equation}

\subsection{Solving the continuity and electron energy equations with the implicit scheme}
In this section, we provide an SG-based implicit scheme  for solving the continuity and electron energy equations with $p=n+1$ in \eqref{eq:tuie}. According to \eqref{eq:fluxee}, the electron temperature is given by
$T_{e,i}^{n+1}=\frac{2E_{e,i}^{n+1}}{3k n_{e,i}^{n+1}}$.
To retain a linear formulation, $n_{e,i}^{n+1}$ is approximated using \eqref{eq:tuie} and \eqref{eq:definition} with $p=n$ and a sufficiently small time step $t_0$. The fluxes are discretized according to \eqref{eq:sgje12}-\eqref{eq:rg12}. Therefore, \eqref{eq:tuie} can be rewritten as
\begin{equation}
\mbox{W}_{i}^{n+1}+\frac{\Delta t}{\Delta x} \cdot (\mbox{F}_{i + \frac{1}{2}}^{n + 1} - \mbox{F}_{i - \frac{1}{2}}^{n + 1})=W_{i}^{n}-\Delta t \cdot S_{i}^{n}. \label{eq:generals}
\end{equation}
Based on \eqref{eq:generals}, we can compute $W^{n+1}$ in \eqref{eq:definition} by solving a linear system
\begin{equation}
\textbf{C}^{n}W^{n+1}=\textbf{R}^{n}, \label{eq:lsystem}
\end{equation}
with $\textbf{C}^{n}=\begin{bmatrix}
\textbf{C}_{1}^{n} & 0 & 0 \\ 
0 & \textbf{C}_{2}^{n} & 0 \\ 
0 & 0 & \textbf{C}_{3}^{n}
\end{bmatrix} \in \mathbb{R}^{3N\times 3N}$ and $\textbf{R}^{n} \in \mathbb{R}^{3N}$. In \eqref{eq:lsystem}, according to \eqref{eq:tuie} and \eqref{eq:definition}, we have
\begin{equation} \label{eq:equation}
\textbf{R}_{i}^{n}=\begin{bmatrix}
\textbf{R}_{1,i}^{n+1} \\
\textbf{R}_{2,i}^{n+1} \\
\textbf{R}_{3,i}^{n+1} \\
\end{bmatrix}=\begin{bmatrix}
n_{e,i}^{n}+\Delta t \cdot S_{e,i}^{n}+\frac{\Delta t}{\Delta x} \cdot R_{1,i}^{n} \\
n_{+,i}^{n}+\Delta t \cdot S_{+,i}^{n}+\frac{\Delta t}{\Delta x} \cdot R_{2,i}^{n} \\
\frac{2E_{e,i}^{n}}{3kn_{e,i}^{n+1}}-\Delta t \cdot \frac{2}{3kn_{e,i}^{n+1}} \cdot H_{r}S_{e,i}^{n}+\frac{\Delta t}{\Delta x} \cdot \frac{2}{3kn_{e,i}^{n+1}} \cdot R_{3,i}^{n} \\
\end{bmatrix}.
\end{equation}
Here, $\textbf{C}_1$, $\textbf{C}_2$ and $\textbf{C}_3 \in \mathbb{R}^{N\times N}$ are block matrices associated with $n_e$, $n_+$ and $T_e$, respectively, while $R_1$, $R_2$, and $R_3$ denote the corresponding contributions of $\mbox{F}$ to the vector on the RHS, $\textbf{R}$. In the following, we construct $\textbf{C}^n$ and $\textbf{R}^n$ for both the interior and the boundary-adjacent cells.

\subsubsection{Interior cells}
Firstly, we construct $\mathbf{C}^{n}$ and $\mathbf{R}^{n}$ for the interior cells, \textit{i.e.}, $1 \le i \le N-2$ in \eqref{eq:generals}. Using \eqref{eq:definition}, \eqref{eq:sgje12}-\eqref{eq:rg12} and \eqref{eq:generals}, the $l$-th row with $2 \le l \le N-1$ of each block matrix $\mathbf{C}_1^n$, $\mathbf{C}_2^n$ and $\mathbf{C}_3^n$ has the nonzero entries
$c_{l,l-1}^{n}
=-\frac{\Delta t}{\Delta x} \cdot 
\frac{\partial \mbox{F}_{l-\frac{3}{2}}^{n}}{\partial W_{l-2}^{n}}$, 
$
c_{l,l}^{n}
=1+\frac{\Delta t}{\Delta x} \cdot 
\left(
\frac{\partial \mbox{F}_{l-\frac{1}{2}}^{n+1}}{\partial W_{l-1}^{n}}
-\frac{\partial \mbox{F}_{l-\frac{3}{2}}^{n+1}}{\partial W_{l-1}^{n}}
\right)$
and 
$c_{l,l+1}^{n}
=\frac{\Delta t}{\Delta x} \cdot 
\frac{\partial \mbox{F}_{l-\frac{1}{2}}^{n}}{\partial W_{l}^{n}}$
in the form of
\begin{equation} \label{eq:jacobil}
c_{l,l-1}^{n}=-\frac{\Delta t}{\Delta x} \cdot \begin{bmatrix}
\frac{D_{e}}{\Delta x} \cdot B(-P_{e,l-\frac{3}{2}}^{n+1}) \\
\frac{D_{+}}{\Delta x} \cdot B(-P_{+,l-\frac{3}{2}}^{n+1}) \\
\frac{2}{3kn_{e,l-1}^{n+1}} \cdot (\frac{3}{2}kn_{e,l-\frac{3}{2}}^{u,n} \cdot \frac{D_{e}}{\Delta x} \cdot B(-P_{T,l-\frac{3}{2}}^{n}))
\end{bmatrix},
\end{equation}
\begin{equation} \label{eq:jacobim} 
c_{l,l}^{n}=1+\frac{\Delta t}{\Delta x} \cdot \begin{bmatrix}
\frac{D_{e}}{\Delta x} \cdot (B(-P_{e,l-\frac{1}{2}}^{n+1})+B(P_{e,l-\frac{3}{2}}^{n+1})) \\
\frac{D_{+}}{\Delta x} \cdot (B(-P_{+,l-\frac{1}{2}}^{n+1})+B(P_{+,i-\frac{3}{2}}^{n+1})) \\
\frac{2}{3kn_{e,l-1}^{n+1}} \cdot (\frac{3}{2}kn_{e,l-\frac{1}{2}}^{u,n} \cdot \frac{D_{e}}{\Delta x} \cdot B(-P_{T,l-\frac{1}{2}}^{n})+\frac{3}{2}kn_{e,l-\frac{3}{2}}^{u,n} \cdot \frac{D_{e}}{\Delta x} \cdot B(P_{T,l-\frac{3}{2}}^{u,n}))
\end{bmatrix},
\end{equation}
\begin{equation} \label{eq:jacobir}
c_{l,l+1}^{n}=\frac{\Delta t}{\Delta x} \cdot \begin{bmatrix}
-\frac{D_{e}}{\Delta x} \cdot B(P_{e,l-\frac{1}{2}}^{n+1}) \\
-\frac{D_{+}}{\Delta x} \cdot B(P_{+,l-\frac{1}{2}}^{n+1}) \\
\frac{2}{3kn_{e,l-1}^{n+1}} \cdot (-\frac{3}{2}kn_{e,l-\frac{1}{2}}^{u,n} \cdot \frac{D_{e}}{\Delta x} \cdot B(P_{T,l-\frac{1}{2}}^{n}))
\end{bmatrix}. 
\end{equation}
\eqref{eq:jacobil}-\eqref{eq:jacobir} for $n_e^{n+1}$, $n_+^{n}$ and $T_e^{n}$ are used to construct $\mathbf{C}_1^n$, $\mathbf{C}_2^n$ and $\mathbf{C}_3^n$, respectively. The vector on the RHS in \eqref{eq:lsystem}, $\mathbf{R}^n$, follows from \eqref{eq:equation} as
\begin{align} 
\textbf{R}_{l}^{n} &=\begin{bmatrix}
\textbf{R}_{1,l}^{n+1} \\
\textbf{R}_{2,l}^{n+1} \\
\textbf{R}_{3,l}^{n+1} \\
\end{bmatrix} =\begin{bmatrix}
n_{e,l-1}^{n}+\Delta t \cdot S_{e,l-1}^{n} \\
n_{+,l-1}^{n}+\Delta t \cdot S_{+,l-1}^{n} \\
K_{0} \\
\end{bmatrix}, \label{eq:rl} \\
\begin{split} \label{eq:k0}
K_{0} &=T_{e,l-1}^{n}-\Delta t \cdot \frac{2}{3kn_{e,l-1}^{n+1}} \cdot H_{r}S_{e,l-1}^{n}-\Delta t \cdot \frac{2}{3kn_{e,l-1}^{n+1}} \cdot e\phi_{l-1}^{n+1} \cdot \frac{\mbox{J}_{e,l-\frac{1}{2}}^{n}-\mbox{J}_{e,l-\frac{3}{2}}^{n}}{\Delta x}  \\ &+\frac{\Delta t}{\Delta x} \cdot \frac{2}{3kn_{e,l-1}^{n+1}} \cdot (e\phi_{l-\frac{1}{2}}^{u,n+1} \cdot \mbox{J}_{e,l-\frac{1}{2}}^{n}-e\phi_{l-\frac{3}{2}}^{u,n+1} \cdot \mbox{J}_{e,l-\frac{3}{2}}^{n}). 
\end{split}
\end{align}

\subsubsection{Boundary-adjacent cells}
Here, we consider the boundary-adjacent cells, corresponding to $i=0$ and $i=N-1$ in \eqref{eq:generals}, which are treated separately.

\textbf{i=0} 

We take $i=0$ in \eqref{eq:generals} as
\begin{equation} \label{eq:generall}
\mbox{W}_{0}^{n+1}+\frac{\Delta t}{\Delta x} \cdot (\mbox{F}_{\frac{1}{2}}^{n + 1} - \mbox{F}_{L}^{n + 1})=W_{0}^{n}-\Delta t \cdot S_{0}^{n}. 
\end{equation}
Because $n_e^{n+1}$ and $n_+^{n+1}$ are solved simultaneously after obtaining $\phi^{n+1}$, using \eqref{eq:jl} and \eqref{eq:glr}, the flux on the left boundary in \eqref{eq:generall} is
$\mbox{F}_{L}^{n+1}
=
\begin{bmatrix}
\mbox{J}_{e,L}^{n+1} \\
\mbox{J}_{+,L}^{n+1} \\
\mbox{G}_{L}^{n+1}
+e\phi_{0}^{n+1}\mathrm{J}_{e,L}^{n+1}
\end{bmatrix}.$
Here, $\mbox{J}_{e,L}^{n+1}$ depends on $n_{e,0}^{n+1}$, $n_{+,0}^{n}$, $\phi_{0}^{n+1}$ and $\phi_{L}^{n+1}$. $\mbox{J}_{+,L}^{n+1}$ depends on $n_{+,0}^{n+1}$, $\phi_{0}^{n+1}$, and $\phi_{L}^{n+1}$. $\mbox{G}_{L}^{n+1}+e\phi_{0}^{n+1}\mbox{J}_{e,L}^{n+1}$ depends on $\kappa_{e,0}^{n}$, $P_{e,L}^{n}$, $T_{e,L}^{n+1}$, $T_{e,0}^{n+1}$, $\phi_{L}^{n+1}$, $\phi_{0}^{n+1}$ and $\mbox{J}_{e,L}^{n+1}$.
%\begin{equation} \label{eq:fl}
%\mbox{F}_{L}^{n+1} =\begin{bmatrix}
%\left\{
%\begin{array}{ll}
%-K_{s}n_{e,0}^{n+1}+2\gamma\mu_{+}n_{+,0}^{n} \cdot \frac{\phi_{0}^{n+1} - \phi_{L}^{n+1}}{\Delta x}, & \phi_{0}^{n+1} \ge \phi_{L}^{n+1} \\
%-K_{s}n_{e,0}^{n+1}, & \phi_{0}^{n+1}<\phi_{L}^{n+1}
%\end{array},
%\right. \\
%\left\{
%\begin{array}{ll}
%-2\mu_{+}n_{+,0}^{n+1} \cdot \frac{\phi_{0}^{n+1} - \phi_{L}^{n+1}}{\Delta x}, & \phi_{0}^{n+1} \ge \phi_{L}^{n+1} \\
%0, & \phi_{0}^{n+1}<\phi_{L}^{n+1}
%\end{array},
%\right. \\
%\frac{2\kappa_{e,0}^{n}}{\Delta x} \cdot \left[B\!\left(-P_{e,L}^{n}\right) T_{e,L}^{n+1}-B\!\left(P_{e,L}^{n}\right)T_{e,0}^{n+1}\right]-e\,\phi_L^{n+1}\,\mbox{J}_{e,L}^{n}+e\phi_{0}^{n+1}\mbox{J}_{e,L}^{n} \\
%\end{bmatrix}.
%\end{equation}
In this way, for $i=0$, \eqref{eq:definition}, \eqref{eq:sgje12}-\eqref{eq:rg12} and \eqref{eq:generall} yield the following nonzero entries in the first row of $\mathbf{C}_1^n$, $\mathbf{C}_2^n$, and $\mathbf{C}_3^n$: $c_{1,1}^{n}=
1+\frac{\Delta t}{\Delta x} \cdot 
\left(
\frac{\partial \mathrm{F}_{\frac{1}{2}}^{n+1}}{\partial W_{0}^{n+1}}
-
\frac{\partial \mathrm{F}_{L}^{n+1}}{\partial W_{0}^{n+1}}
\right)$ and 
$c_{1,2}^{n}
=
\frac{\Delta t}{\Delta x} \cdot 
\frac{\partial \mathrm{F}_{\frac{1}{2}}^{n+1}}{\partial W_{1}^{n+1}}$
where
\begin{equation} \label{eq:c11} 
c_{1,1}^{n} =1+\frac{\Delta t}{\Delta x} \cdot \begin{bmatrix}
\frac{D_{e}}{\Delta x} \cdot B(-P_{e,\frac{1}{2}}^{n+1})+K_{s} \\
\frac{D_{+}}{\Delta x} \cdot B(-P_{+,\frac{1}{2}}^{n+1})+\frac{\partial \mbox{J}_{+,L}^{n+1}}{\partial n_{+,0}^{n+1}} \\
\frac{2}{3kn_{e,0}^{n+1}} \cdot (\frac{3}{2}kn_{e,\frac{1}{2}}^{u,n} \cdot \frac{D_{e}}{\Delta x} \cdot B(-P_{T,\frac{1}{2}}^{n})+\frac{2\kappa_{e,0}^{n}}{\Delta x} \cdot B(P_{e,L}^{n}))
\end{bmatrix}, 
\end{equation}
\begin{equation} \label{eq:c12}
c_{1,2}^{n} =\frac{\Delta t}{\Delta x} \cdot \begin{bmatrix}
-\frac{D_{e}}{\Delta x} \cdot B(P_{e,\frac{1}{2}}^{n+1}) \\
-\frac{D_{+}}{\Delta x} \cdot B(P_{+,\frac{1}{2}}^{n+1}) \\
\frac{2}{3kn_{e,0}^{n+1}} \cdot (-\frac{3}{2}kn_{e,\frac{1}{2}}^{u,n} \cdot \frac{D_{e}}{\Delta x} \cdot B(P_{T,\frac{1}{2}}^{n}))
\end{bmatrix}. 
\end{equation}
Accordingly, the vector on the RHS of \eqref{eq:lsystem}, $\mathbf{R}^n$, follows from  \eqref{eq:equation} as
\begin{equation} \label{eq:rhs1} 
\textbf{R}_{1}^{n} =\begin{bmatrix}
\textbf{R}_{1,1}^{n+1} \\
\textbf{R}_{2,1}^{n+1} \\
\textbf{R}_{3,1}^{n+1} \\
\end{bmatrix}=\begin{bmatrix}
n_{e,0}^{n}+\Delta t \cdot S_{e,0}^{n}-\frac{\Delta t}{\Delta x} \cdot \gamma\mbox{J}_{+,L}^{n} \\
n_{+,0}^{n}+\Delta t \cdot S_{+,0}^{n} \\
K_{1} 
\end{bmatrix}, \\
\end{equation}
\begin{equation} \label{eq:k1}
\begin{split}
K_{1} &=T_{e,0}^{n}-\Delta t \cdot \frac{2}{3kn_{e,0}^{n+1}} \cdot H_{r}S_{e,0}^{n} +\Delta t \cdot \frac{2}{3kn_{e,0}^{n+1}} \cdot e\phi_{0}^{n+1} \cdot \frac{\mbox{J}_{e,L}^{n}-\mbox{J}_{e,\frac{1}{2}}^{n}}{\Delta x} \\ &+\frac{\Delta t}{\Delta x} \cdot \frac{2}{3kn_{e,0}^{n+1}} \cdot (e\phi_{L}^{n+1}\mbox{J}_{e,L}^{n}-\frac{2\kappa_{e,0}^{n}}{\Delta x} \cdot B(-P_{e,L}^{n})T_{e,L}^{n+1}-e\phi_{\frac12}^{u,n+1}\mbox{J}_{e,\frac{1}{2}}^{n}).
\end{split}
\end{equation}

\textbf{i=N-1} 

We take $i=N-1$ in \eqref{eq:generals} as
\begin{equation} \label{eq:generalr}
W_{N-1}^{n + 1} + \frac{\Delta t}{\Delta x} \cdot (\mbox{F}_{R}^{n + 1} - \mbox{F}_{N-\frac{3}{2}}^{n + 1})=W_{N-1}^{n}-\Delta t \cdot S_{N-1}^{n + 1}. 
\end{equation}
Analogous to the case when $i=0$, for $i=N-1$ with \eqref{eq:generalr}, based on \eqref{eq:jr} and \eqref{eq:grr}, we have
$\mbox{F}_{R}^{n+1}
=
\begin{bmatrix}
\mbox{J}_{e,R}^{n+1} \\
\mbox{J}_{+,R}^{n+1} \\
\mbox{G}_{R}^{n+1}
+e\phi_{N-1}^{n+1}\mbox{J}_{e,R}^{n+1}
\end{bmatrix}$.
Specifically, $\mbox{J}_{e,R}^{n+1}$ depends on $n_{e,N-1}^{n+1}$, $n_{+,N-1}^{n}$, $\phi_{R}^{n+1}$ and $\phi_{N-1}^{n+1}$, whereas $\mbox{J}_{+,R}^{n+1}$ depends on $n_{+,N-1}^{n+1}$, $\phi_{R}^{n+1}$ and $\phi_{N-1}^{n+1}$. Moreover, $\mbox{G}_{R}^{n+1}+e\phi_{N-1}^{n+1}\mbox{J}_{e,R}^{n+1}$ depends on $\kappa_{e,N-1}^{n}$, $P_{e,R}^{n}$, $T_{e,N-1}^{n+1}$, $T_{e,R}^{n+1}$, $\phi_{R}^{n+1}$, $\phi_{N-1}^{n+1}$ and $\mbox{J}_{e,R}^{n+1}$.
%\begin{equation} \label{eq:fr}
%\mbox{F}_{R}^{n+1} =\begin{bmatrix}
%\left\{
%\begin{array}{ll}
%K_{s}n_{e,N-1}^{n+1}+2\gamma\mu_{+}n_{+,N-1}^{n} \cdot \frac{\phi_{R}^{n+1} - \phi_{N-1}^{n+1}}{\Delta x}, & \phi_{R}^{n+1} \le \phi_{N-1}^{n+1} \\
%K_{s}n_{e,N-1}^{n+1}, & \phi_{R}^{n+1}>\phi_{N-1}^{n+1}
%\end{array},
%\right. \\
%\left\{
%\begin{array}{ll}
%-2\mu_{+}n_{+,N-1}^{n+1} \cdot \frac{\phi_{R}^{n+1} - \phi_{N-1}^{n+1}}{\Delta x}, & \phi_{R}^{n+1} \le \phi_{N-1}^{n+1} \\
%0, & \phi_{R}^{n+1}>\phi_{N-1}^{n+1}
%\end{array},
%\right. \\
%\frac{2\kappa_{e,N-1}^{n}}{\Delta x} \cdot %\left[B\!\left(-P_{e,R}^{n}\right) T_{e,N-1}^{n+1}-%B\!\left(P_{e,R}^{n}\right) T_{e,R}^{n+1}\right]-%e\,\phi_R^{n+1}\,\mbox{J}_{e,R}^{n}+e\phi_{N-%1}^{n+1}\mbox{J}_{e,R}^{n} \\
%\end{bmatrix}.
%\end{equation}
Therefore, for $i=N-1$, \eqref{eq:definition}, \eqref{eq:sgje12}-\eqref{eq:rg12} and \eqref{eq:generalr} yield the following nonzero entries in the $N$-th row of $\mathbf{C}_1^n$, $\mathbf{C}_2^n$, and $\mathbf{C}_3^n$:
$c_{N,N}^{n}=
1+\frac{\Delta t}{\Delta x} \cdot 
\left(
\frac{\partial \mathrm{F}_{R}^{n+1}}{\partial W_{N-1}^{n+1}}-
\frac{\partial \mathrm{F}_{N-\frac{3}{2}}^{n+1}}{\partial W_{N-1}^{n+1}}
\right)$ and $c_{N,N-1}^{n}=
\frac{\Delta t}{\Delta x} \cdot 
\frac{\partial \mathrm{F}_{N-\frac{3}{2}}^{n+1}}{\partial W_{N-2}^{n+1}}$
where
\begin{equation} \label{eq:c1n}
c_{N,N}^{n} =1+\frac{\Delta t}{\Delta x} \cdot \begin{bmatrix}
K_{s}+\frac{D_{e}}{\Delta x} \cdot B(P_{e,N-\frac{3}{2}}^{n+1}) \\
\frac{\partial \mbox{J}_{+,R}^{n+1}}{\partial n_{+,N-1}^{n+1}}+\frac{D_{+}}{\Delta x} \cdot B(P_{+,N-\frac{3}{2}}^{n+1}) \\
\frac{2}{3kn_{e,N-1}^{n+1}} \cdot (\frac{2\kappa_{e,N-1}^{n}}{\Delta x} \cdot B(-P_{e,R}^{n})+\frac{3}{2}kn_{e,N-\frac{3}{2}}^{u,n} \cdot \frac{D_{e}}{\Delta x} \cdot B(P_{T,N-\frac{3}{2}}^{n}))
\end{bmatrix}, \\
\end{equation}
\begin{equation} \label{eq:c2n}
c_{N,N-1}^{n} =\frac{\Delta t}{\Delta x} \cdot \begin{bmatrix}
-\frac{D_{e}}{\Delta x} \cdot B(-P_{e,N-\frac{3}{2}}^{n+1}) \\
-\frac{D_{+}}{\Delta x} \cdot B(-P_{+,N-\frac{3}{2}}^{n+1}) \\
\frac{2}{3kn_{e,N-1}^{n+1}} \cdot (-\frac{3}{2}kn_{e,N-\frac{3}{2}}^{u,n} \cdot \frac{D_{e}}{\Delta x} \cdot B(-P_{T,N-\frac{3}{2}}^{n}))
\end{bmatrix}. 
\end{equation}
Consequently, the vector on the RHS of \eqref{eq:lsystem}, $\textbf{R}^{n}$, follows from \eqref{eq:equation} as
\begin{align}
\begin{split} \label{eq:rhsn}
\textbf{R}_{N-1}^{n} &=\begin{bmatrix}
\textbf{R}_{1,N-1}^{n+1} \\
\textbf{R}_{2,N-1}^{n+1} \\
\textbf{R}_{3,N-1}^{n+1} \\
\end{bmatrix} =\begin{bmatrix}
n_{e,N-1}^{n}+\Delta t \cdot S_{e,N-1}^{n}+\frac{\Delta t}{\Delta x} \cdot \gamma\mbox{J}_{+,R}^{n} \\
n_{+,N-1}^{n}+\Delta t \cdot S_{+,N-1}^{n} \\
K_{2} 
\end{bmatrix},
\end{split} \\
\begin{split}
K_{2} &=T_{e,N-1}^{n} -\Delta t \cdot \frac{2}{3kn_{e,N-1}^{n+1}} \cdot H_{r}S_{e,N-1}^{n}+\Delta t \cdot \frac{2}{3kn_{e,N-1}^{n+1}} \cdot e\phi_{N-1}^{n+1} \cdot \frac{\mbox{J}_{e,N-\frac{3}{2}}^{n}-\mbox{J}_{e,R}^{n}}{\Delta x} \\ &+\frac{\Delta t}{\Delta x} \cdot \frac{2}{3kn_{e,N-1}^{n+1}} \cdot (-e\phi_{N-\frac{3}{2}}^{u,n+1}\mbox{J}_{e,N-\frac{3}{2}}^{n}+\frac{2\kappa_{e,N-1}^{n}}{\Delta x} \cdot B(P_{e,R}^{n})T_{e,R}^{n+1}+e\phi_{R}^{n+1}\mbox{J}_{e,R}^{n}). \label{eq:k2}
\end{split}
\end{align}

\subsubsection{Summary of $\mathbf{C}_i^n W^{n+1}=\mathbf{R}_i^n$}
In summary, applying \eqref{eq:equation} to the linear system \eqref{eq:lsystem} yields the coefficient matrix $\mathbf{C}^n$ and the vector on the RHS, $\mathbf{R}_i^n$. The entries of $\mathbf{C}^n$ are specified by \eqref{eq:jacobil}-\eqref{eq:jacobir}, \eqref{eq:c11}, \eqref{eq:c12}, \eqref{eq:c1n} and \eqref{eq:c2n}, while $\mathbf{R}_i^n$ is given by \eqref{eq:rl}, \eqref{eq:rhs1}, \eqref{eq:k1}, \eqref{eq:rhsn} and \eqref{eq:k2}.

\subsection{Numerical algorithm of 1D-CCP}
The numerical algorithm for the simulation of 1D-CCP is shown in \cref{alg:1D-CCP}.

\begin{algorithm}
\caption{Numerical algorithm of 1D-CCP}
\label{alg:1D-CCP}
\begin{algorithmic}[1] % 加上 [1] 可以自動顯示行號，更具可讀性
\REQUIRE Initial values $n_e^0, n_+^0, T_e^0, \phi^{0}$, steps $\Delta t, \Delta x$, $T=0$, $n=0$, $T_{\max}$.
\ENSURE Final states $n_e^T, n_+^T, T_e^T, \phi^T$.

\WHILE{$T < T_{\max}$}
    \STATE \textbf{Step 1:} Solve the Poisson's equation with \eqref{eq:ax=b} to obtain $\phi^{n+1}$.
    \STATE \textbf{Step 2:} Solve the other equations \eqref{eq:lsystem} using $\phi^{n+1}$ to obtain $n_e^{n+1}, n_+^{n+1}, T_e^{n+1}$.
    \STATE \textbf{Step 3:} Update time and indices:
    \STATE \quad $T \leftarrow T + \Delta t$
    \STATE \quad $n \leftarrow n + 1$
\ENDWHILE
\end{algorithmic}
\end{algorithm}

\section{Numerical analysis of the schemes}
\label{sec:analysis}
This section presents a detailed numerical analysis of the schemes introduced in \cref{sec:schemes}, focusing on the discrete particle balance, positivity of the solutions and the uniform bounds for the variables. A consistency analysis can be carried out by standard Taylor expansion.

\subsection{Discrete particle balance}
Firstly, we establish the discrete particle balance for the densities of the particles, $n_e$ and $n_{+}$.

\begin{theorem}[Discrete particle balance of $n_{e}$ and $n_{+}$]\label{theorem:mass}
For $n \in \{n_e,n_{+}\}$, let $n^{n+1}$ be computed by the implicit scheme in \eqref{eq:generals} with $\mbox{J}_L=\{\mbox{J}_{+,L},\mbox{J}_{e,L}\}=\mbox{J}_{-\frac{1}{2}}$ and $\mbox{J}_R=\{\mbox{J}_{+,R},\mbox{J}_{e,R}\}=\mbox{J}_{N-\frac{1}{2}}$. Then, the scheme satisfies the following discrete particle balance
\begin{equation}
\Delta x \cdot \sum_{i=0}^{N-1}n_{+}^{n+1}-\Delta x \cdot \sum_{i=0}^{N-1}n_{+}^{n}=\Delta t\Delta x \cdot \sum_{i=0}^{N-1}S_{e}^{n}-\Delta t(\mbox{J}_{R}^{n+1}-\mbox{J}_{L}^{n+1}). \label{eq:conservation}
\end{equation}
\end{theorem}
\begin{proof}
With \eqref{eq:generals}, we have
\begin{equation}
\frac{n_{+}^{n+1}-n_{+}^{n}}{\Delta t}+\frac{1}{\Delta x} \cdot (\mbox{J}_{i+\frac{1}{2}}^{n+1}-\mbox{J}_{i-\frac{1}{2}}^{n+1})=S_{e,i}^{n}. \label{eq:equivalent}
\end{equation}
Multiplying \eqref{eq:equivalent} by $\Delta t\,\Delta x$ and summing over $i=0,\ldots,N-1$, with $\mbox{J}_L=\{\mbox{J}_{+,L},\mbox{J}_{e,L}\}=\mbox{J}_{-\frac{1}{2}}$ and $\mbox{J}_R=\{\mbox{J}_{+,R},\mbox{J}_{e,R}\}=\mbox{J}_{N-\frac{1}{2}}$, \eqref{eq:conservation} holds. This completes the proof of \cref{theorem:mass}.
\end{proof}

\subsection{Uniqueness and positivity of the solutions}
For a fixed spatial mesh size $\Delta x$, the time step $\Delta t$ affects both the computational efficiency and the key qualitative properties of the numerical solutions. Therefore, we derive conditions on $\Delta t$ that guarantee the uniqueness of $n_e$, $n_{+}$, $T_e$ and $\phi$, as well as the positivity of the physically relevant variables $n_e$, $n_{+}$, and $T_e$. We take $E^{n+1}$ because $\phi^{n+1}$ is computed before considering $n_{e}^{n+1}$, $n_{+}^{n+1}$ and $T_{e}^{n+1}$. $n_{e}^{n+1}$ is approximated with $n_{e}^{n}$ explicitly as we have mentioned.

\begin{theorem}[$\Delta t$ for keeping the physical properties of the variables]\label{theorem t}
For the time step $\Delta t$, when we take
\begin{align}
\Delta t &< \min({\frac{\Delta x}{\mu_{e} \cdot \max \limits_{1 \le i \le N-2}|E_{i+\frac{1}{2}}^{n+1}-E_{i-\frac{1}{2}}^{n+1}|}, \frac{\Delta x}{|\mu_{e} \cdot E_{\frac{1}{2}}^{n+1}-K_{s}|},\frac{\Delta x}{|\mu_{e} \cdot E_{N-\frac{3}{2}}^{n+1}+K_{s}|}}), \label{eq:t1} 
\end{align}
\begin{align} 
\Delta t &<\min({\frac{\Delta x}{\mu_{+} \cdot \max \limits_{1 \le i \le N-2}|{E_{i-\frac{1}{2}}^{n+1}-E_{i+\frac{1}{2}}^{n+1}|}}, \frac{\Delta x}{|\mu_{+} \cdot E_{\frac{1}{2}}^{n+1}+\frac{\partial J_{+,L}^{n+1}}{\partial n_{+,0}^{n+1}}|},\frac{\Delta x}{|\mu_{+} \cdot E_{N-\frac{3}{2}}^{n+1}-\frac{\partial J_{+,R}^{n+1}}{\partial n_{+,N-1}^{n+1}}|}}), \label{eq:t2} \\
\Delta t&<\min(t_{1}, t_{2}, t_{3}), \label{eq:t3} 
\end{align}
with $E_{i+\frac{1}{2}}=-\frac{\phi_{i+1}-\phi_{i}}{\Delta x}$, with \eqref{eq:t1}-\eqref{eq:t3} in the computation of $n_{e}^{n+1}, n_{+}^{n+1}, T_{e}^{n+1}, \phi^{n+1}$ by solving \eqref{eq:ax=b} and \eqref{eq:lsystem}, they are always unique for $n \ge 0$. If $n_{e}^{0}, n_{+}^{0}, T_{e}^{0}>0$, we also have $n_{e}^{n}, n_{+}^{n}, T_{e}^{n}>0$. 
In \eqref{eq:t3}, we define
\begin{align} 
t_{1} &=\min(\frac{3kn_{e,0}^{n+1} \cdot (\Delta x)^{2}}{|5k\mbox{J}_{e,\frac{1}{2}}^{n} \cdot \Delta x+4\kappa_{e,0}^{n}B(P_{e,L}^{n})|}, \quad \frac{3kn_{e,0}^{n+1}T_{e,0}^{n} \cdot \Delta x}{|2H_{r}S_{e,0}^{n} \cdot \Delta x-a_{1}|}), \label{eq:at} \\
t_{2} &=\min(\frac{3n_{e,i}^{n+1} \cdot \Delta x}{5D_{e} \cdot \max \limits_{1 \le i \le N-2}|\mbox{J}_{e,i-\frac{1}{2}}^{n}-\mbox{J}_{e,i+\frac{1}{2}}^{n}|}, \quad \frac{3kn_{e,l-1}^{n+1}T_{e,l-1}^{n} \cdot \Delta x}{|2H_{r}S_{e,l-1}^{n} \cdot \Delta x-a_{2}|}), \label{eq:bt} \\
t_{3} &=\min(\frac{3kn_{e,N-1}^{n+1} \cdot (\Delta x)^{2}}{|5k\mbox{J}_{e,N-\frac{3}{2}}^{n}-4\kappa_{e,N-1}^{n}B(-P_{e,R}^{n})|}, \quad \frac{3kn_{e,N-1}^{n+1}T_{e,N-1}^{n} \cdot \Delta x}{|2H_{r}S_{e,N-1}^{n} \cdot \Delta x-a_{3}|}), \label{eq:ct}
\end{align}
with $a_{1}=2(e\phi_{0}^{n+1} \cdot (\mbox{J}_{e,L}^{n}-\mbox{J}_{e,\frac{1}{2}}^{n})-\frac{2\kappa_{e,0}^{n}}{\Delta x} \cdot B(-P_{e,L}^{n})T_{e,L}^{n+1}-e\phi_{\frac{1}{2}}^{u,n+1}\mbox{J}_{e,\frac{1}{2}}^{n}+e\phi_{L}^{n+1} \cdot \mbox{J}_{e,L}^{n})$, 
$a_{2}=2e(\phi_{l-\frac{1}{2}}^{u,n+1} \cdot \mbox{J}_{e,l-\frac{1}{2}}^{n}-\phi_{l-\frac{3}{2}}^{u,n+1} \cdot \mbox{J}_{e,l-\frac{3}{2}}^{n}-\phi_{l-1}^{n+1} \cdot \mbox{J}_{e,l-\frac{1}{2}}^{n}+\phi_{l-1}^{n+1} \cdot \mbox{J}_{l-\frac{3}{2}}^{n})$ and $a_{3}=2(e\phi_{N-1}^{n+1} \cdot (\mbox{J}_{e,N-\frac{3}{2}}^{n}-\mbox{J}_{e,R}^{n})-e\phi_{N-\frac{3}{2}}^{u,n+1}\mbox{J}_{e,N-\frac{3}{2}}^{n}+\frac{2\kappa_{e,N-1}^{n}}{\Delta x} \cdot B(P_{e,R}^{n})T_{e,R}^{n+1}+e\phi_{R}^{n+1}\mbox{J}_{e,R}^{n})$.
\end{theorem}
\begin{proof}
We focus on the computation of $\phi^{n+1}$ with \eqref{eq:ax=b} first. It is clear that \eqref{eq:ax=b} is equivalent to
\begin{equation}
-\textbf{A}^{n}\phi^{n+1}=-\textbf{b}^{n}. \label{eq:ax=be}
\end{equation}
It follows from \eqref{eq:iitemsg}, \eqref{eq:ax=b1a}, and \eqref{eq:ax=b5} that the tri-diagonal matrix $-\textbf{A}^{n} \in \mathbb{R}^{N\times N}$ is symmetric, with strictly diagonally dominant first and last rows and weakly diagonally dominant interior rows. The elements on the diagonal of $-\textbf{A}^{n}$ are all positive. Hence, $-\mathbf{A}^{n}$ is positive definite and non-singular, \textit{i.e.}, $\mbox{rank}(-\textbf{A}^{n})=N$. Therefore, \eqref{eq:ax=be} admits a unique solution $\phi^{n+1}$.

To compute $n_e^{n+1}$, $n_+^{n+1}$ and $T_e^{n+1}$ from \eqref{eq:lsystem}, we solve the coupled linear system
$\textbf{C}^{n}W^{n+1}=\mathbf{R}^{n}$. A sufficient condition for the existence, uniqueness and the positivity of the solution is that $\textbf{C}^{n}$ be diagonally dominant, with positive entries on the diagonal and non-positive entries on the off-diagonals. Under this condition, $\mathbf{C}^{n}$ is non-singular based on the Gershgorin circle theorem \cite{yuanpan}. It is also an $M$-matrix satisfying $(\textbf{C}^{n})^{-1} \ge 0$. Moreover, since $\textbf{C}^{n}$ is invertible, no row of $(\textbf{C}^{n})^{-1}$ is identically zero. Therefore, if $\textbf{R}^{n}>0$, then $W^{n+1}>0$. These conditions impose restrictions on the time step $\Delta t$. In view of \eqref{eq:equation}, the conditions must be verified simultaneously for $\textbf{C}_1^n$ and $\textbf{R}_1^n$ associated with $n_e^{n+1}$, $\textbf{C}_2^n$ and $\textbf{R}_2^n$ associated with $n_+^{n+1}$, and $\textbf{C}_3^n$ and $\textbf{R}_3^n$ associated with $T_e^{n+1}$. The proof proceeds by mathematical induction, assuming $n_e^0,n_+^0,T_e^0>0$.

Assume that $n_e^n>0$, $n_+^n>0$, and $T_e^n>0$. By \eqref{eq:rl}, \eqref{eq:rhs1} and \eqref{eq:rhsn}, if $k_{r,i}^n>0$, then $S_{e,i}^n=S_{+,i}^n=k_{r,i}^nN_0n_{e,i}^n \ge 0$, 
which implies $\textbf{R}_2^n>0$. Moreover, \eqref{eq:jl} and \eqref{eq:jr} give $\mbox{J}_{+,L}^n \le 0$ and $\mbox{J}_{+,R}^n \ge 0$, and hence $\textbf{R}_1^n>0$. From \eqref{eq:jacobil}-\eqref{eq:jacobir}, \eqref{eq:c11}, \eqref{eq:c12}, \eqref{eq:c1n} and \eqref{eq:c2n}, the diagonal entries of $\textbf{C}_1^n$, $\textbf{C}_2^n$ and $\textbf{C}_3^n$ are positive, whereas their off-diagonal entries are negative, provided that $B(x)$, $D_e$, $D_+$ and $K_s$ are positive. Unlike $\textbf{R}_1^n$ and $\textbf{R}_2^n$, the positivity of $\textbf{R}_3^n$ requires a suitable restriction on $\Delta t$. Therefore, if $\Delta t$ is chosen such that $\textbf{C}_1^n$, $\textbf{C}_2^n$, and $\textbf{C}_3^n$ are diagonally dominant and $\textbf{R}_3^n>0$, then
\begin{equation}
\label{eq:tiao} % 這裡是用於整個公式群組的總編號
\begin{cases}
\abs{c_{1,1}^{n}}>\abs{c_{1,2}^{n}} \\
\abs{c_{l,l}^{n}}>\abs{c_{l,l-1}^{n}}+\abs{c_{l,l+1}^{n}} \quad 2 \le l \le N-2 \\
\abs{c_{N,N}^{n}}>\abs{c_{N,N-1}^{n}}
\end{cases}, \quad
\textbf{R}_{3}^{n}=\begin{bmatrix}
K_{0} \\
K_{1} \\
K_{2} \\
\end{bmatrix}>0, 
\end{equation}
where $K_0$, $K_1$, and $K_2$ are defined in \eqref{eq:k0}, \eqref{eq:k1} and \eqref{eq:k2}. The conditions in \eqref{eq:tiao} ensure the positivity of $n_e^{n+1}$, $n_+^{n+1}$ and $T_e^{n+1}$. Solving the inequalities in \eqref{eq:tiao} yields the time-step restrictions \eqref{eq:t1}-\eqref{eq:t3}, together with \eqref{eq:at}-\eqref{eq:ct}. Under these conditions, \eqref{eq:tiao} holds, with \eqref{eq:p} and \eqref{eq:setb} implying
$n_e^{n+1}>0, n_+^{n+1}>0, T_e^{n+1}>0$. Thus, the induction argument is complete.

In summary, if $\Delta t>0$ satisfies \eqref{eq:t1}-\eqref{eq:t3} together with \eqref{eq:at}-\eqref{eq:ct}, then, for all $n \ge 0$, the variables $n_e^{n+1}$, $n_+^{n+1}$, $T_e^{n+1}$ and $\phi^{n+1}$ are uniquely determined. Moreover, if $n_{e}^{0}, n_{+}^{0}, T_{e}^{0}>0$, then $n_{e}^{n}>0, n_{+}^{n}>0, T_{e}^{n}>0$. This completes the proof of \cref{theorem t}.
\end{proof}

\subsection{Uniform bounds of the variables}
In this part, we establish uniform upper bounds for $n_e$, $n_{+}$, $E_e$, $T_e$ and $\phi$, assuming that the reaction coefficient $k_r$ in \eqref{eq:source} is uniformly bounded and that $n_e$, $n_{+}$ and $T_e$ remain positive. Here, the electron energy $E_e$ is defined in \eqref{eq:fluxee}.

\begin{theorem}[Uniform bounds of the variables]\label{theorem:bound}
For the CCP simulation, we set the maximum time $T_{max}$ and the initial values of the variables as
\begin{equation}
\norm{n_{e}^{0}}_{1}=F_{1}, \quad \norm{n_{+}^{0}}_{1}=F_{2}, \quad \norm{\phi^{0}}_{1}=F_{3}, \quad \norm{E_{e}^{0}}_{1}=F_{4}, \label{eq:initial}
\end{equation}
with $E_{e,i}^{n}=\frac{3}{2}k \cdot n_{e,i}^{n}T_{e,i}^{n}$, $0 \le i \le N-1$. $k_{r}$ in \eqref{eq:source} satisfies $0<k_{b} \le k_{r}^{n} \le k_{0}$ with $k_{b}, k_{0}$ its uniform bounds. When $\Delta t$ is small enough in the first and second time step, satisfying $(n+1)\Delta t \le T_{max}$, \eqref{eq:t1}-\eqref{eq:t3},
\begin{equation} \label{eq:t4}
\Delta t \le \min(t_{4}, t_{5}, t_{6}),
\end{equation}
with $t_{4}=\frac{\Delta x}{|S_{e,0}^{n}\Delta x-(\mbox{r}_{e,\frac{1}{2}}^{n}-\mbox{J}_{e,L}^{n})|}, t_{5}=\frac{\Delta x}{\max \limits_{1 \le i \le N-2}{|S_{e,i}^{n}\Delta x-(\mbox{r}_{e,i+\frac{1}{2}}^{n}-\mbox{r}_{e,i-\frac{1}{2}}^{n})|}}$ and \\ $t_{6}=\frac{\Delta x}{|S_{e,N-1}^{n}\Delta x-(\mbox{J}_{e,R}^{n}-\mbox{r}_{e,N-\frac{3}{2}}^{n})|}$, $n_{e}, n_{+}, \phi, E_{e}, T_{e}$ has the uniform bounds as
\begin{equation}
\norm{n_{e}^{n}}_{1} \le U_{e}, \quad \norm{n_{+}^{n}}_{1} \le U_{+}, \quad \norm{\phi^{n}}_{2} \le U_{\phi}, \quad n_{e,i}^{n} \ge L, \quad \norm{E_{e}^{n}}_{1} \le U_{E}, \quad \norm{T_{e}^{n}}_{1} \le U_{T}, \label{eq:upper}
\end{equation}
with $0 \le i \le N-1$. $U_{e}, U_{+}, U_{\phi}, L, U_{E}, U_{T}$ in \eqref{eq:upper} are defined in \eqref{eq:bound1}, \eqref{eq:bound3}, \eqref{eq:bound2}, \eqref{eq:l}, \eqref{eq:bound4} and \eqref{eq:ut}.
\end{theorem}

\begin{proof}
The same as the proof of \cref{theorem t}, we focus on these variables separately.

\textbf{Uniform bound of $n_{e}$}
We consider $n_e$ first. Under the restrictions of the time step in \eqref{eq:t1}-\eqref{eq:t3}, $n_e^n$ and $n_{+}^n$ remain positive. Summing \eqref{eq:generals} over $i=0,\cdots,N-1$ and multiplying by $\Delta t$ yields
\begin{equation} \label{eq:sume}
\sum_{i=0}^{N-1} n_{e,i}^{n+1}-\sum_{i=0}^{N-1} n_{e,i}^{n}+\frac{\Delta t}{\Delta x} \cdot (\mbox{J}_{e,R}^{n+1}-\mbox{J}_{e,L}^{n+1})=\Delta t \cdot \sum_{i=0}^{N-1}S_{e,i}^{n}.
\end{equation}
According to \eqref{eq:jl} and \eqref{eq:jr}, given that $n_{e}^{n+1}>0$ has been established via \eqref{eq:t1}-\eqref{eq:t3}, it holds for $K_{s}>0$ that
\begin{equation}
\mbox{J}_{e,R}^{n+1}-\mbox{J}_{e,L}^{n+1}=K_{s}n_{e,N-1}^{n+1}-\gamma\mbox{J}_{+,R}^{n+1}+K_{s}n_{e,0}^{n+1}+\gamma \mbox{J}_{+,L}^{n+1} \ge -\gamma \cdot (\mbox{J}_{+,R}^{n+1}-\mbox{J}_{+,L}^{n+1}). \label{eq:differ}
\end{equation}
Therefore, combining \eqref{eq:sum} and \eqref{eq:differ}, we obtain
\begin{equation} \label{eq:in}
\sum_{i=0}^{N-1} n_{e,i}^{n+1} \le \sum_{i=0}^{N-1} n_{e,i}^{n}+\gamma \cdot \frac{\Delta t}{\Delta x} \cdot (\mbox{J}_{+,R}^{n}-\mbox{J}_{+,L}^{n})+\Delta t \cdot \sum_{i=0}^{N-1}S_{e,i}^{n}.
\end{equation}
For $n_{+}$, with \eqref{eq:generals}, we can have
\begin{equation} \label{eq:transfer1}
\frac{\Delta t}{\Delta x} \cdot (\mbox{J}_{+,R}^{n}-\mbox{J}_{+,L}^{n})=\sum_{i=0}^{N-1} n_{+,i}^{n-1}-\sum_{i=0}^{N-1} n_{+,i}^{n}+\Delta t \cdot \sum_{i=0}^{N-1}S_{e,i}^{n-1}.
\end{equation}
Following \eqref{eq:source}, the source term is discretized using the explicit scheme
$S_{e,i}^{n}=k_{r,i}^{n}N_{0}n_{e,i}^{n}$. Therefore, with $n_{+}>0$ and $k_{r}^{n} \le k_{0}$, we put \eqref{eq:transfer1} into \eqref{eq:in} and we can have
\begin{equation} \label{eq:ne+n+}
\begin{split}
\sum_{i=0}^{N-1} n_{e,i}^{n+1} &+\gamma 
\cdot \sum_{i=0}^{N-1} n_{+,i}^{n} \\ \le \sum_{i=0}^{N-1} n_{e,i}^{n} &+\gamma \cdot \sum_{i=0}^{N-1} n_{+,i}^{n-1}+\Delta t \cdot \sum_{i=0}^{N-1}S_{e,i}^{n}+\gamma \cdot \Delta t \cdot \sum_{i=0}^{N-1}S_{e,i}^{n-1} \\ \le \sum_{i=0}^{N-1} n_{e,i}^{n} &+\gamma \cdot \sum_{i=0}^{N-1} n_{+,i}^{n-1}+\Delta t \cdot \sum_{i=0}^{N-1}k_{r.i}^{n}N_{0}n_{e,i}^{n}+\gamma \cdot \Delta t \cdot \sum_{i=0}^{N-1}k_{r,i}^{n-1}N_{0}n_{e,i}^{n-1} \\ \le \sum_{i=0}^{N-1} n_{e,i}^{n} &+\gamma \cdot \sum_{i=0}^{N-1} n_{+,i}^{n-1}+\Delta t \cdot k_{0}N_{0} \cdot \sum_{i=0}^{N-1}n_{e,i}^{n}+\gamma \cdot \Delta t \cdot k_{0}N_{0} \cdot \sum_{i=0}^{N-1}n_{e,i}^{n-1} \\ \le \sum_{i=0}^{N-1} n_{e,i}^{n} &+\gamma \cdot \sum_{i=0}^{N-1} n_{+,i}^{n-1}+\Delta t \cdot k_{0}N_{0} \cdot (\sum_{i=0}^{N-1} n_{e,i}^{n}+\gamma \cdot \sum_{i=0}^{N-1} n_{+,i}^{n-1}) \\ &+\gamma \cdot \Delta t \cdot k_{0}N_{0} \cdot (\sum_{i=0}^{N-1} n_{e,i}^{n-1}+\gamma \cdot \sum_{i=0}^{N-1} n_{+,i}^{n-2}).
\end{split}
\end{equation}
Let $V^{n}=\sum \limits_{i=0}^{N-1} n_{e,i}^{n}+\gamma \cdot \sum \limits_{i=0}^{N-1} n_{+,i}^{n-1}$, $m_{1} = 1+\Delta t k_{0}N_{0}$ and $m_{2} = \gamma \Delta t k_{0}N_{0}$, where $V^{n}, m_{1}, m_{2}>0$ for $n \ge 1$. Consequently, \eqref{eq:ne+n+} reduces to
\begin{equation} \label{eq:control}
V^{n+1}-m_{1}V^{n}-m_{2}V^{n-1} \le 0. 
\end{equation}
Solving the characteristic equation $r^{2}-m_{1}r-m_{2}=0$ yields two roots
\begin{equation} 
r_{1} =\frac{m_{1}+\sqrt{m_{1}^{2}+4m_{2}}}{2} =\frac{(1+\Delta t \cdot k_{0}N_{0})+\sqrt{1+2(1+2\gamma)k_{0}N_{0}\Delta t+(k_{0}N\Delta t)^{2}}}{2}>0, 
\label{eq:r1} 
\end{equation}
\begin{equation}
r_{2} =\frac{m_{1}-\sqrt{m_{1}^{2}+4m_{2}}}{2} =\frac{(1+\Delta t \cdot k_{0}N_{0})-\sqrt{1+2(1+2\gamma)k_{0}N_{0}\Delta t+(k_{0}N_{0}\Delta t)^{2}}}{2}<0. \label{eq:r2}
\end{equation}
Applying Taylor's formula, $\sqrt{1+x}=1+\frac{1}{2}x+\mathcal{O}(x^{2})$, to \eqref{eq:r1} and \eqref{eq:r2}, we obtain
\begin{align}
r_{1} &=\frac{1+\Delta t \cdot k_{0}N_{0}+1+(1+2\gamma)k_{0}N_{0}\Delta t}{2}=1+(1+\gamma) \cdot k_{0}N_{0}\Delta t>0, \label{eq:r1r} \\
r_{2} &=\frac{1+\Delta t \cdot k_{0}N_{0}-1-(1+2\gamma)k_{0}N_{0}\Delta t}{2}=-\gamma \cdot k_{0}N_{0}\Delta t<0. \label{eq:r2r} 
\end{align}
By substituting \eqref{eq:r1r} and \eqref{eq:r2r} into \eqref{eq:control}, applying the resulting inequality inductively $n-1$ times yields
\begin{equation} \label{eq:td1}
V^{n+1} \le V^{n+1}-r_{2}V^{n} \le r_{1}(V^{n}-r_{2}V^{n-1})  \le r_{1}^{n-1}(V^{2}-r_{2}V^{1}).
\end{equation}
When $(n-1)\Delta t \le T_{max}$, with \eqref{eq:r1r}, we have
\begin{equation}
r_{1}=1+(1+\gamma) \cdot k_{0}N_{0}\Delta t \le 1+(1 +\gamma) \cdot k_{0}N_{0} \cdot \frac{T_{max}}{n-1}=r_{1b}. \label{eq:r1b}
\end{equation}
Given $\norm{n_{e}^{n+1}}_{1}=\sum \limits_{i=0}^{N-1}n_{e,i}^{n+1}$, $n_{+}>0$ and $\norm{n_{+}^{0}}_{1}=F_{2}$ from \eqref{eq:initial} and \eqref{eq:r1b}, substituting the definition of $V^{n}$ along with \eqref{eq:r1r} and \eqref{eq:r2r} into \eqref{eq:td1} yields
\begin{equation} \label{eq:equation1}
\begin{split}
\norm{n_{e}^{n+1}}_{1} 
&\le \sum \limits_{i=0}^{N-1} n_{e,i}^{n+1} +\gamma \cdot \sum \limits_{i=0}^{N-1} n_{+,i}^{n} \\ 
&\le r_{1}^{n-1}(\sum \limits_{i=0}^{N-1} n_{e,i}^{2} +\gamma \cdot \sum \limits_{i=0}^{N-1} n_{+,i}^{1}-r_{2}(\sum \limits_{i=0}^{N-1} n_{e,i}^{1}+\gamma \cdot \sum \limits_{i=0}^{N-1} n_{+,i}^{0})) \\ 
&\le r_{1b}^{n-1} \cdot (\sum \limits_{i=0}^{N-1} n_{e,i}^{2} +\gamma \cdot \sum \limits_{i=0}^{N-1} n_{+,i}^{1}-r_{2}(\sum \limits_{i=0}^{N-1} n_{e,i}^{1}+\gamma \cdot \sum \limits_{i=0}^{N-1} n_{+,i}^{0})) \\ 
&\le \underbrace{e^{(1+\gamma) \cdot k_{0}N_{0}T_{max}} \cdot (\sum \limits_{i=0}^{N-1} n_{e,i}^{2}+\gamma \cdot \sum \limits_{i=0}^{N-1} n_{+,i}^{1}-r_{2}(\sum \limits_{i=0}^{N-1} n_{e,i}^{1}+\gamma \cdot F_{2}))}_{Y}.
\end{split}
\end{equation}
with $n \ge 2$. It is evident that
$Y>\sum_{i=0}^{N-1} n_{e,i}^{2}$.
For the initial time step with $n=0$, in addition to the restrictions \eqref{eq:t1}-\eqref{eq:t4}, we choose a sufficiently small $\Delta t$ such that \eqref{eq:tiao} holds. Then, for the linear system \eqref{eq:lsystem}, 
$\min \limits_{1\leq i\leq N}\left(c_{i,i}^{1}-c_{i,i-1}^{1}-c_{i,i+1}^{1}\right)\geq 0.05,
\textbf{R}^{1}\leq \textbf{W}^{0}+C_R\textbf{1}$,
where $\textbf{1}=(1,1,\cdots,1)^{\top}$ and $C_R>0$ is a constant. Since $\textbf{C}^{1}$ is an $M$-matrix \cite{1937}, it follows that
$\left\|(\textbf{C}^{1})^{-1}\right\|_2 \le 20$.
Consequently, we can get 
$\left\|\textbf{W}^{1}\right\|_2
\leq 20\left\|\textbf{W}^{0}+C_R\textbf{1}\right\|_2$. Therefore, $\|\textbf{W}^{1}\|_2$ is uniformly bounded. In view of the positivity of $n_e$, $n_+$ and $T_e$, both
$\|n_e^1\|_1=\sum \limits_{i=0}^{N-1}n_{e,i}^1$ and 
$\|n_+^1\|_1=\sum \limits_{i=0}^{N-1}n_{+,i}^1$
are bounded. By an analogous argument, $\|n_e^2\|_1$ is also bounded for a sufficiently small $\Delta t$. Therefore, together with \eqref{eq:equation1}, $\|n_e^n\|_1$ admits a uniform upper bound
\begin{equation}
U_{e}=\max(Y, \norm{n_{e}^{1}}_{1}, F_{1}), \label{eq:bound1}
\end{equation}
with $n \ge 0$. Here, $Y$ is defined in \eqref{eq:equation1}. 

\textbf{Uniform bound of $n_{+}$}
Here, we consider the component $n_{+}$. Summing both sides of \eqref{eq:generals} from $i=0$ to $N-1$ and multiplying by $\Delta t$ yields
\begin{equation} \label{eq:sum1}
\sum_{i=0}^{N-1} n_{+,i}^{n+1}-\sum_{i=0}^{N-1} n_{+,i}^{n}+\frac{\Delta t}{\Delta x} \cdot (\mbox{J}_{+,R}^{n+1}-\mbox{J}_{+,L}^{n+1})=\Delta t \cdot \sum_{i=0}^{N-1}S_{e,i}^{n}.
\end{equation}
According to \eqref{eq:jl} and \eqref{eq:jr}, when $n_{+}^{n+1}>0$, it follows that $\mbox{J}_{+,R}^{n+1} \ge 0$ and $\mbox{J}_{+,L}^{n+1} \le 0$, which implies $\mbox{J}_{+,R}^{n+1}-\mbox{J}_{+,L}^{n+1} \ge 0$. Consequently, from \eqref{eq:sum1}, we obtain
\begin{equation} \label{eq:fang}
\sum_{i=0}^{N-1} n_{+,i}^{n+1} \le \sum_{i=0}^{N-1} n_{+,i}^{n+1}+\frac{\Delta t}{\Delta x} \cdot (\mbox{J}_{+,R}^{n+1}-\mbox{J}_{+,L}^{n+1}) \le \sum_{i=0}^{N-1} n_{+,i}^{n}+\Delta t \cdot \sum_{i=0}^{N-1}S_{e,i}^{n}.
\end{equation}
Therefore, when $n_{+}^{n}>0$, given $(n+1)\Delta t \le T_{\max}$, $\norm{n_{+}^{0}}_{1}=F_{2}$ from \eqref{eq:initial}, $\norm{n_{e}^{n}}_{1} \le U_{e}$, and $k_{r,i}^{n} \le k_{0}$, iterating \eqref{eq:fang} $n+1$ times establishes an upper bound for $\norm{n_{+}^{n+1}}_{1}$ as
\begin{equation} \label{eq:boundn+}
\begin{split}
\norm{n_{+}^{n+1}}_{1} &\le \sum_{i=0}^{N-1} n_{+,i}^{0}+(n+1)\Delta t \cdot \sum_{k=0}^{n}\sum_{i=0}^{N-1}S_{e,i}^{k} \\ &\le \sum_{i=0}^{N-1} n_{+,i}^{0}+(n+1)\Delta t \cdot \sum_{i=0}^{N-1}k_{r,i}^{k}N_{0}n_{e,i}^{k} \\ 
&\le \sum_{i=0}^{N-1} n_{+,i}^{0}+(n+1) \cdot \Delta t \cdot k_{0}N_{0} \cdot U_{e} \\  & \le F_{2}+T_{max}k_{0}N_{0}U_{e}.
\end{split}
\end{equation}
The estimate \eqref{eq:boundn+} holds for all $n \ge 0$ satisfying $(n+1)\Delta t \le T_{\max}$. In view of \eqref{eq:initial}, the RHS of \eqref{eq:boundn+} strictly exceeds $F_{2}=\norm{n_{+}^{0}}_{1}$. In this way, $\norm{n_{+}^{n}}_{1}$ is bounded from above by
\begin{equation}
U_{+}=F_{2}+T_{max}k_{0}N_{0}U_{e}, \label{eq:bound3}
\end{equation}
with $n \ge 0$.

\textbf{Uniform bound of $\phi$}
Since $\phi^{n}$ is computed by solving \eqref{eq:ax=be},  according to \eqref{eq:iitemsg}, \eqref{eq:ax=b1a} and \eqref{eq:ax=b5}, we can rewrite $-\textbf{A}$ in the form of $-\textbf{A}^{n}=\textbf{M}+\textbf{D}^{n}$ with
\begin{equation} \label{eq:sum}
\textbf{M} = \frac{\epsilon_{0}}{(\Delta x)^{2}} \cdot \begin{bmatrix}
 3 & -1 &        &        &    \\
-1 &  2 & -1     &        &    \\
   & -1 & \ddots & \ddots &    \\
   &    & \ddots &  2     & -1 \\
   &    &        & -1     &  3 
\end{bmatrix}, \quad \textbf{D}_{i,j}^{n} = \begin{cases} 
\frac{e\Delta t}{(\Delta x)^{2}} \cdot (A_{e,i-\frac{3}{2}}^{n}+A_{e,i-\frac{1}{2}}^{n}) \quad & i = j \\ 
\frac{e\Delta t}{(\Delta x)^{2}} \cdot A_{e,i-\frac{1}{2}}^{n} \quad & j = i + 1 \\ 
\frac{e\Delta t}{(\Delta x)^{2}} \cdot A_{e,i-\frac{3}{2}}^{n} \quad & j = i - 1 \\ 
0, & \text{otherwise} 
\end{cases}. 
\end{equation}
By \eqref{eq:sum}, $\textbf{M} \in \mathbb{R}^{N\times N}$ is symmetric and strictly diagonally dominant in its first and last rows, while it is weakly diagonally dominant elsewhere. Similarly, $\textbf{D}^n\in\mathbb{R}^{N\times N}$ is symmetric and weakly diagonally dominant in every row. Both matrices have positive entries on the diagonal and non-positive on the off-diagonals. Consequently, $-\textbf{A}^n$ and $\textbf{M}$ are positive definite, whereas $\textbf{D}^n$ is positive semi-definite. In particular,
$(-\textbf{A}^n)^{-1}>0$. For the RHS of \eqref{eq:ax=be}, specified in \eqref{eq:iitemsg}, \eqref{eq:ax=b1a} and \eqref{eq:ax=b5}, suppose that $n_{e,i} \le \norm{n_{e}}_{1}^n \le U_e$ and $n_{+,i} \le \norm{n_{+,i}}_{1}^n \le U_+$. If $\Delta t$ satisfies \eqref{eq:t4}, then \eqref{eq:bb} implies that
\begin{equation} \label{eq:bbound}
b_{l-1}^{n} \le  \left\{
\begin{array}{ll}
e \cdot (U_{e}+U_{+}+1)+\frac{2\epsilon_{0}\phi_{RF}}{(\Delta x)^{2}}, &l=2 \\
e \cdot (U_{e}+U_{+}+1), & 3 \le l \le N+1 
\end{array}. 
\right.
\end{equation}
Therefore, with \eqref{eq:bbound}, we can have
\begin{equation} \label{eq:bb1}
\norm{-\textbf{b}^{n}}_{2} \le \sqrt{(e \cdot (U_{e}+U_{+}+1)+\frac{2\epsilon_{0}\phi_{RF}}{(\Delta x)^{2}})^{2}+(e \cdot (U_{e}+U_{+}+1))^{2} \cdot (N-1)}=H_{b}.
\end{equation}
Here, $\textbf{H}_{b}$ is a positive constant. Given that the matrix $(-\textbf{A}^{n})^{-1}$ is positive definite, we define its spectral norm as $B_{1}=\norm{(-\textbf{A}^{n})^{-1}}_{2}>0$. Consequently, combining \eqref{eq:ax=be} and \eqref{eq:bb1}, we can bound $\norm{\phi^{n+1}}_{1}$ for $n \ge 0$ as
\begin{equation}
\norm{\phi^{n+1}}_{2} \le \norm{(-\textbf{A}^{n})^{-1}}_{2} \cdot \norm{-\textbf{b}^{n}}_{2} \le B_{1}H_{b}. \label{eq:phin1}
\end{equation}
Recall from \eqref{eq:initial} that $\norm{\phi^{0}}_{2}=F_{3}$. In view of \eqref{eq:phin1}, we conclude that $\norm{\phi^{n}}_{2}$ is bounded from above by
\begin{equation}
U_{\phi}=\max(B_{1}H_{b}, F_{3}), \label{eq:bound2}
\end{equation}
with $n \ge 0$. Here, $B_{1}H_{b}$ is defined in \eqref{eq:phin1}.

\textbf{Uniform bound of $n_{e,i}$}
To establish the lower bound for $n_{e,i}^{n}$, we focus on the Bernoulli function. According to \eqref{eq:initial}, $\abs{\phi_{i}^{n}} \le \norm{\phi^{n}}_{2} \le U_{\phi}$. Therefore, from \eqref{eq:p}, we can bound $\abs{P_{e,i+\frac{1}{2}}}$ as
\begin{equation}
\abs{P_{e,i+\frac{1}{2}}} \le \mu_{e} \cdot \frac{\phi_{i+1}-\phi_{n}}{D_{e}} \le \frac{2\mu_{e}U_{\phi}}{D_{e}}. \label{eq:boundp}
\end{equation}
Since the Bernoulli function $B(x)=\frac{x}{e^{x}-1}>0$ is continuous and strictly decreasing with respect to $x$, combining this property with \eqref{eq:boundp} allows us to bound $B(P_{e})$ as
\begin{equation}
B(P_{e}) \le B(-\abs{P_{e}}) \le \frac{-2\mu_{e}U_{\phi}}{D_{e}(e^{-\frac{2\mu_{e}U_{\phi}}{D_{e}}}-1)}=B_{u}. \label{eq:bu}
\end{equation}
Supposing that $\Delta t$ satisfies \eqref{eq:t1}--\eqref{eq:t4}, and under the conditions \eqref{eq:source}, \eqref{eq:rl} and $k_{b} \le k_{r}^{n} \le k_{0}$, we obtain
\begin{equation} \label{eq:inl}
\textbf{C}_{1}^n n_e^{n+1}=(I_{n}+\Delta t N_{0}K^{n})n_{e}^{n} \ge (1+\Delta tN_{0}k_{b})n_{e}^{n},
\end{equation}
with $K^n=\mbox{diag} (k_{r,0}^n, k_{r,1}^{n}, \dots, k_{r,N-1}^n)$. Based on \eqref{eq:jacobil}-\eqref{eq:jacobir}, \eqref{eq:c11}, \eqref{eq:c12}, \eqref{eq:c1n} and \eqref{eq:c2n}, recalling that the bound $B(P_{e}) \le B_{u}$ in \eqref{eq:bu} has been established under the constraints pf the time step \eqref{eq:t1}-\eqref{eq:t4}, we can get
\begin{equation}
\textbf{C}_{1}^n \mathbf{1}\le (1+\Delta t\Lambda)\mathbf{1}, \label{eq:cbound}
\end{equation}
where $\Lambda=\min\left\{
\frac{4D_eB_u}{(\Delta x)^2},
\frac{2D_eB_u}{(\Delta x)^2}+\frac{K_s}{\Delta x}
\right\}>0$ is independent of $n$, $i$ and $\Delta t$. In the following, we prove, by mathematical induction, that $n_e^n \ge n_e^0 e^{-(\Lambda-k_b)n\Delta t}$.
Let $m=e^{-(\Lambda-k_b)\Delta t}>0$ and assume that
$n_e^n\geq n_e^0m^n\mathbf{1}$. The assertion is immediate for $n=0$. Suppose it holds at the $n$-th time step. Then, at the $(n+1)$-th time step, \eqref{eq:inl} yields
\begin{equation} \label{eq:relation1}
\textbf{C}_{1}^{n} n_{e}^{n+1} \ge (1+\Delta t k_{b})n_{e}^{n} \ge (1+\Delta t k_{b})m^n\mathbf1.
\end{equation}
In view of \eqref{eq:cbound}, provided that $1+\Delta t k_{b}>0$, we also get
\begin{equation} \label{eq:relation2}
\textbf{C}_{1}^{n} (\frac{1+\Delta t k_{b}}{1+\Delta t\Lambda} \cdot m^n\mathbf1) \le (1+\Delta t k_{b}) \cdot m^{n}\mathbf1.
\end{equation}
When $\Delta t$ satisfies \eqref{eq:t1}-\eqref{eq:t4}, $\textbf{C}_{1}^{n}$ is an $M$-matrix, which implies that its inverse is non-negative, \textit{i.e.}, $(\textbf{C}_{1}^{n})^{-1} \ge 0$, and contains no zero rows. Consequently, from \eqref{eq:relation1} and \eqref{eq:relation2}, we obtain
\begin{equation} \label{eq:relation}
n_{e}^{n+1} \ge (\frac{1+\Delta t k_{b}}{1+\Delta t\Lambda} \cdot m^n\mathbf1).
\end{equation}
By setting $m=\frac{1+\Delta t k_{b}}{1+\Delta t\Lambda}$, the induction step is complete. Consequently, we can choose
\begin{equation}
n_{e}^{n} \ge n_{e}^{0}(\frac{1+\Delta t k_{b}}{1+\Delta t\Lambda})^{n} \cdot\mathbf1. \label{eq:nen}
\end{equation}
From \eqref{eq:nen}, if $k_b\geq\Lambda$, then
$\left(\frac{1+\Delta t\,k_b}{1+\Delta t\,\Lambda}\right)^n\geq 1$, and hence $n_e^n \ge n_e^0$.
Therefore, $n_{e,i}^n$ admits the uniform lower bound $n_{e,i}^0$ for $0 \le i \le N-1$. If $k_b<\Lambda$, we define $f(x)=\log(1+\Delta t\,x)$ for $x\geq 0$. Since
$f'(x)=\frac{\Delta t}{1+\Delta t\,x}\leq\Delta t$, 
the mean value theorem gives
\begin{equation} \label{eq:mean}
\log(\frac{1+\Delta t k_{b}}{1+\Delta t\Lambda}) =-(\log(1+\Delta t\Lambda)-\log(1+\Delta t k_{b})) =\frac{\Delta t}{1+\Delta t \cdot \xi} \cdot (-(\Lambda-k_{b})) \ge -\Delta t \cdot (\Lambda-k_{b}),
\end{equation}
where $\xi \in (k_{b}, \Lambda)$. Combining \eqref{eq:mean} with the complementary case $k_{b} \ge \Lambda$, we arrive at
\begin{equation} \label{eq:function}
\frac{1+\Delta tk_{b}}{1+\Delta t\Lambda} \ge e^{-\Delta t(\Lambda-k_{b})_{+}}.
\end{equation}
Here, $x_{+}=\max(x,0)$. For $n \Delta t \le T_{\max}$, substituting \eqref{eq:function} into \eqref{eq:nen} yields
\begin{equation} \label{eq:result}
n_{e}^{n} =n_{e}^{0}(\frac{1+\Delta t k_{b}}{1+\Delta t\Lambda})^{n} \cdot\mathbf1 \ge n_{e}^{0}e^{-n\Delta t(\Lambda-k_{b})_+} \cdot \mathbf{1} \ge n_{e}^{0} e^{-T_{max}(\Lambda-k_{b})_+} \cdot \mathbf{1}.
\end{equation}
In this way, in view of \eqref{eq:result}, we conclude that $n_{e,i}^{n}$ is uniformly bounded by
\begin{equation} \label{eq:l}
L=n_{e,i}^{0} e^{-T_{max}(\Lambda-k_{b})_+}, 
\end{equation}
for $0 \le i \le N-1$, where $\Lambda = \min \left\{ \frac{4D_{e}B_{u}}{(\Delta x)^{2}}, \frac{2D_{e}B_{u}}{(\Delta x)^{2}}+\frac{K_{s}}{\Delta x} \right\}$.

\textbf{Uniform bounds of $E_{e}$ and $T_{e}$}
Next, we consider $E_{e}$ and $T_{e}$. To begin with, we establish bounds for the fluxes. In view of \eqref{eq:sgje12} and \eqref{eq:bu}, when $n_{e}^{n}>0$ and $\norm{n_{e}^{n}}_{1} \le U_{e}$, the flux $\abs{\mbox{J}_{e,i+\frac{1}{2}}}$ can be bounded as
\begin{equation}
\abs{\mbox{J}_{e,i+\frac{1}{2}}} \le \frac{D_{e}}{\Delta x} \cdot (B_{u}U_{e}+B_{u}U_{e})=\frac{2D_{e}B_{u}U_{e}}{\Delta x}=B_{e}. \label{eq:be}
\end{equation}
Given $n_{+}^{n}>0$, $\norm{n_{e}^{n}}_{1} \le U_{e}$, $\norm{n_{+}^{n}}_{1} \le U_{+}$ and $\norm{\phi^{n}}_{2} \le U_{\phi}$, it follows from \eqref{eq:bb}, \eqref{eq:jl} and \eqref{eq:jr} that the fluxes $\abs{\mbox{J}_{e,L}}$ and $\abs{\mbox{J}_{+,L}}$ are bounded by
\begin{equation}
\abs{\mbox{J}_{e,L}} \le K_{s}U_{e}+\gamma \cdot 2\mu_{+}U_{+} \cdot (U_{\phi}+\phi_{RF}) \cdot \frac{1}{\Delta x}=\frac{K_{s}U_{e}\Delta x+2\gamma\mu_{+}U_{+} \cdot (U_{\phi}+\phi_{RF})}{\Delta x} =B_{L}, \label{eq:bl} 
\end{equation}
\begin{equation}
\abs{\mbox{J}_{e,R}} \le K_{s}U_{e}+\gamma \cdot 2\mu_{+}U_{+} \cdot U_{\phi} \cdot \frac{1}{\Delta x} =\frac{K_{s}U_{e}\Delta x+2\gamma\mu_{+}U_{+} \cdot U_{\phi}}{\Delta x} =B_{R}. \label{eq:br}
\end{equation}
In view of \eqref{eq:bl} and \eqref{eq:br}, when $n_{e,i}^{n} \ge L$, the terms $P_{e,L}$ and $P_{e,R}$ can be bounded as
\begin{equation}
\abs{P_{e,L}} \le \frac{5B_{L} \cdot \Delta x}{6D_{e} \cdot L}=B_{pL}, \quad \abs{P_{e,R}} \le \frac{5B_{R} \cdot \Delta x}{6D_{e} \cdot L}=B_{pR}. \label{eq:pe}
\end{equation}
Similarly to \eqref{eq:bu}, utilizing \eqref{eq:pe} allows us to bound $B(P_{e,L}^{n})$ and $B(P_{e,R}^{n})$ as
\begin{equation}
B(P_{e,L}) \le B(-\abs{P_{e,j}})=\frac{-B_{pL}}{e^{-B_{pL}}-1}=B_{bL}, \quad B(P_{e,R}) \le B(-P_{e,R})=\frac{-B_{pR}}{e^{-B_{pR}}-1}=B_{bR}. \label{eq:bbb}
\end{equation}
With these bounds, we focus on $E_{e}$. We rewrite \eqref{eq:generals} as
\begin{equation} \label{eq:eegeneral}
\frac{E_{e,i}^{n+1}-E_{e,i}^{n}}{\Delta t}+\frac{G_{i+\frac{1}{2}}^{n+1}-G_{i-\frac{1}{2}}^{n+1}}{\Delta x}+e\phi_{i}^{n+1} \cdot \frac{\mbox{J}_{e,i+\frac{1}{2}}^{n}-\mbox{J}_{e,i-\frac{1}{2}}^{n}}{\Delta x}=-H_{r}S_{e,i}^{n},
\end{equation}
for $0 \le i \le N-1$. Summing \eqref{eq:eegeneral} over $i$ from $0$ to $N-1$ and multiplying both sides by $\Delta t$, we find that when the source term $H_{r}S_{e,i}^{n}>0$, it follows that
\begin{equation} \label{eq:move}
\sum \limits_{i=0}^{N-1}E_{e,i}^{n+1}-\sum \limits_{i=0}^{N-1}E_{e,i}^{n}+\frac{\Delta t}{\Delta x} \cdot (G_{R}^{n+1}-G_{L}^{n+1}) < -\frac{\Delta t}{\Delta x} \cdot \sum \limits_{i=0}^{N-1}[e\phi_{i}^{n+1} \cdot (\mbox{J}_{e,i+\frac{1}{2}}^{n}-\mbox{J}_{e,i-\frac{1}{2}}^{n})].
\end{equation}
%\begin{equation}
%[e\phi_{N-1}^{n+1} \cdot (\mbox{J}_{e,R}^{n}-\mbox{J}_{e,N-\frac{3}{2}}^{n})+e\phi_{0}^{n+1} \cdot (\mbox{J}_{e,\frac{1}{2}}^{n}-\mbox{J}_{e,L}^{n})+\sum \limits_{i=1}^{N-2} (e\phi_{i}^{n} \cdot (\mbox{J}_{e,i+\frac{1}{2}}^{n}-\mbox{J}_{e,i-\frac{1}{2}}^{n})].
%\end{equation}
With \eqref{eq:bb} and \eqref{eq:glr}-\eqref{eq:setb}, when $n_{e}, T_{e}>0$, we write
$G_{R}^{n+1}-G_{L}^{n+1}=G_{1}^{n+1}+G_{2}^{n+1}$ in \eqref{eq:move} with 
\begin{align}
G_{1}^{n+1} &=\frac{2\kappa_{e,N-1}^{n}}{\Delta x} \cdot B(-P_{e,R}^{n})T_{e,N-1}^{n+1}-\frac{2\kappa_{e,0}^{n}}{\Delta x} \cdot B(P_{e,L}^{n})T_{e,0}^{n+1}>0,
 \label{eq:g1} \\
\begin{split} \label{eq:g2}
G_{2}^{n+1} &=-\frac{2\kappa_{e,N-1}^{n}}{\Delta x} \cdot B(P_{e,R}^{n})T_{e,R}^{n+1}-\frac{2\kappa_{e,0}^{n}}{\Delta x} \cdot B(-P_{e,L}^{n})T_{e,L}^{n+1} \\ &=-\frac{\kappa_{e,N-1}^{n}}{\Delta x} \cdot B(P_{e,R}^{n})-\frac{\kappa_{e,0}^{n}}{\Delta x} \cdot B(-P_{e,L}^{n}). 
\end{split}
\end{align}
Consequently, substituting \eqref{eq:g1} and \eqref{eq:g2} into \eqref{eq:move} yields
\begin{equation} \label{eq:conclusion}
\sum \limits_{i=0}^{N-1}E_{e,i}^{n+1}-\sum \limits_{i=0}^{N-1}E_{e,i}^{n} <-\frac{\Delta t}{\Delta x} \cdot G_{2}^{n+1}-\frac{\Delta t}{\Delta x} \cdot \sum \limits_{i=0}^{N-1}[e\phi_{i}^{n+1} \cdot (\mbox{J}_{e,i+\frac{1}{2}}^{n}-\mbox{J}_{e,i-\frac{1}{2}}^{n})].
\end{equation}
Substituting \eqref{eq:bb}, \eqref{eq:setb}, \eqref{eq:bbb}, \eqref{eq:be}-\eqref{eq:br} and \eqref{eq:g2} into \eqref{eq:conclusion}, under the conditions $n_{e}^{n} \le U_{e}$ and $\phi^{n} \le U_{\phi}$, the RHS of \eqref{eq:conclusion} can be estimated as
\begin{equation} \label{eq:BE}
\begin{split}
&\abs{-\frac{\Delta t}{\Delta x} \cdot \sum \limits_{i=0}^{N-1}[e\phi_{i}^{n+1} \cdot (\mbox{J}_{e,i+\frac{1}{2}}^{n}-\mbox{J}_{e,i-\frac{1}{2}}^{n})]} \\ 
&\le \left.
\begin{aligned}
&\frac{1}{\Delta x} \cdot (\frac{1}{\Delta x} \cdot \frac{3}{2}kD_{e}U_{e} \cdot B_{bR}+\frac{1}{\Delta x} \cdot \frac{3}{2}kD_{e}U_{e} \cdot B_{bL}. \\ 
&+\, e\phi_{RF} \cdot B_{L}+eU_{\phi} \cdot (B_{R}+B_{e})+eU_{\phi}(B_{L}+B_{e})+(N-2) \cdot eU_{\phi} \cdot (B_{e}+B_{e}))
\end{aligned}
\right\} B_{E}.
\end{split}
\end{equation}
Given the initial condition $\norm{E_{e}^{0}}_{1}=\sum_{i=0}^{N-1}E_{e,i}^{0}=F_{4}$ from \eqref{eq:initial}, we substitute \eqref{eq:BE} into \eqref{eq:conclusion}. By iterating the resulting relation inductively $n+1$ times, we obtain
\begin{equation} \label{eq:finalee}
\norm{E_{e}^{n+1}}_{1}\le \sum \limits_{i=0}^{N-1}E_{e,i}^{n}+\Delta t \cdot B_{E} \le \sum \limits_{i=0}^{N-1}E_{e,i}^{0}+(n+1)\Delta t \cdot B_{E} \le F_{4}+T_{max}B_{E},
\end{equation}
for $(n+1)\Delta t \le T_{\max}$, where the initial condition satisfies $\sum \limits_{i=0}^{N-1}E_{e,i}^{0}=\frac{3}{2}k \sum \limits_{i=0}^{N-1}n_{e,i}^{0}T_{e,i}^{0}$ for $n \ge 0$. In this way, from \eqref{eq:finalee}, we conclude that $\norm{E_{e}^{n}}_{1}$ is bounded from above by
\begin{equation}
U_{E}=F_{4}+T_{max}B_{E}, \label{eq:bound4}
\end{equation}
for $n \ge 0$, where $B_{E}$ is defined in \eqref{eq:BE}. To analyze $T_{e}^{n}$, since $E_{e,i}^{n}=\frac{3}{2}kn_{e,i}^{n}T_{e,i}^{n}>0$, the constraint $n_{e,i}^{n} \ge L$ implies that $T_{e,i}^{n}$ is uniformly bounded from above by $\frac{2U_{E}}{3kL}$. Consequently, a uniform upper bound for $\norm{T_{e}^{n}}_{1}$ is given by
\begin{equation}
U_{T}=\frac{2NU_{E}}{3kL}. \label{eq:ut}
\end{equation}

In all, when $\Delta t$ satisfies \eqref{eq:t1}-\eqref{eq:t3} and \eqref{eq:t4}, we always have \eqref{eq:bound1}, \eqref{eq:bound3}, \eqref{eq:bound2}, \eqref{eq:l}, \eqref{eq:bound4} and \eqref{eq:ut}. \eqref{eq:upper} holds. This completes the proof of \cref{theorem:bound}.
\end{proof}

\section{Numerical experiments}
\label{sec:experiments}
This section presents the numerical results of the 1D-CCP using various schemes. We evaluate their accuracy, robustness and the applicable time step over the total simulation time. The parameters for the numerical simulation are specified as follows. The length of the domain for the simulation is $L = 25.4~\mathrm{mm}$, containing the density of the neutral gas $N_0 = 3.2188 \times 10^{22}~\mathrm{m^{-3}}$. The system is driven by an RF voltage with the amplitude $\Phi_{\mathrm{RF}} = 100~\mathrm{V}$ at a radio frequency of $f_{\mathrm{RF}} = 13.56~\mathrm{MHz}$. The coefficients of the transport are $\mu_e = 30.0~\mathrm{m^2\,V^{-1}\,s^{-1}}$, $D_e = 119.87~\mathrm{m^2\,s^{-1}}$, $\mu_{+} = 0.1444~\mathrm{m^2\,V^{-1}\,s^{-1}}$ and $D_{+} = 6.42 \times 10^{-3}~\mathrm{m^2\,s^{-1}}$. At the boundaries, the electron loss coefficient, secondary electron emission coefficient and the energy lost per chemical reaction are defined as $K_s = 1.19 \times 10^5~\mathrm{m\,s^{-1}}$, $\gamma = 0.01$ and $H_{r} = 2.515454 \times 10^{-18}~\mathrm{J}$, respectively. The initial profiles are uniform with $n_{e}^{0} = n_{+}^{0} = \textbf{10}^{\textbf{13}}~\mathrm{m^{-3}}$, $T_{e}^{0} = \textbf{11604}~\mathrm{K}$ and $\phi^{0} = \textbf{0}~\mathrm{V}$. Unless specified otherwise, the domain is discretized into $N = 128$ uniform meshes. Notably, all schemes share the identical physical models, transport datasets, boundary conditions, and the configurations of the linear solver.

\begin{table}
\centering
\caption{Discretization combinations considered in this study. Unless stated
otherwise, SG--Picard denotes the SG-consistent semi-implicit Poisson
treatment including the lagged chemical source term.}
\label{tab:discrete_combinations}
\small
\resizebox{\linewidth}{!}{%
\begin{tabular}{llllll}
\toprule
Scheme & Fluid advance & Particle flux & Poisson treatment &
Energy flux/boundary & Purpose \\
\midrule
\FIPSGEFC & Implicit & SG & SG--Picard & SG-consistent/SG-consistent & Theoretical analysis \\
\FEP & Explicit & Conventional drift--diffusion & Explicit & Conventional/Conventional & Explicit reference \\
\FIP & Implicit & Conventional drift--diffusion & Semi-implicit & Conventional/Conventional & Comparison \\
\FIPSG & Implicit & SG & SG--Picard & Conventional/Conventional & Comparison \\
\FIPSGEBC & Implicit & SG & SG--Picard & Conventional/EBC & Efficiency comparison \\
\bottomrule
\end{tabular}%
}
\end{table} 

The five discretizations considered in this work are summarized in \cref{tab:discrete_combinations}. The \FIPSGEFC discretization is used in the theoretical analysis and in the stability, robustness, and accuracy
studies below. The four related discretizations \FEP, \FIP, \FIPSG and \FIPSGEBC are compared separately in \cref{sec:efficiency} to assess the verified time-step range and total computational cost. \FIPSGEBC retains the conventional interior electron-energy flux while applying the EBC treatment at the walls. All numerical experiments are performed on an Intel Xeon w5-2455X processor using GNU C++ 9.4.0 in Release mode.

% TODO for the final theory--experiment comparison:
% Extend the phasewise bridge to the uniform-boundedness sufficient-condition
% family once its notation and dimensional mapping have been finalized.

\subsection{Numerical stability and robustness}
\label{sec:scheme_assessment}
Before considering the numerical stability and robustness of \FIPSGEFC, we start with the time step $\Delta t$ first. We parameterize the time increment by the number of steps per RF cycle, $N_{\mathrm{RF}}$, as $T_{\mathrm{RF}} = \frac{1}{f_{\mathrm{RF}}}$ and $\Delta t = \frac{T_{\mathrm{RF}}}{N_{\mathrm{RF}}}$. Although the sufficient conditions derived in \eqref{eq:t1}-\eqref{eq:t3} and \eqref{eq:t4} guarantee the uniqueness, positivity and the uniform bounds of the variables, they do not dictate the sharp threshold in the practical simulations. Therefore, the applicable range of $\Delta t$ is evaluated numerically. We take the numerical simulation for $1500$ RF cycles. For \FIPSGEFC, a time step of $\Delta t = 9.60238 \times 10^{-11}~\mathrm{s}$ (corresponding to $N_{\mathrm{RF}} = 768$) is adopted. At this time step, $n_e$, $n_+$, and $T_e$ remain strictly positive throughout this long-term simulation. The solution successfully approaches a finite periodic steady state, validating the conservative and positivity-preserving properties established in the theoretical analysis.

In the following, we evaluate the robustness of \FIPSGEFC across a parameter sweep of $125$ cases combining $p \in \{0.1, 0.3, 0.5, 1, 2\}~\mathrm{Torr}$, $\Phi_{\mathrm{RF}} \in \{100, 120, 140, 160, 200\}~\mathrm{V}$ and $f_{\mathrm{RF}} \in \{6.78, 13.56, 27.12, 40.68, 54.24\}~\mathrm{MHz}$. While the ionization model is retained, transport coefficients are scaled with pressure. Rather than replicating the specific configuration of the main benchmark, this broad parameter sweep provides a stringent robustness assessment. \cref{fig:robustness_map} maps the resulting periodic states across different $p$, with the axes representing $\Phi_{\mathrm{RF}}$ and $f_{\mathrm{RF}}$. Symbols and colors denote the converged states, whereas $\times$ marks the cases failing to reach periodicity within the prescribed window. Out of $125$ cases, $122$ cases satisfy the criterion of the convergence and $120$ cases achieve the sustained discharges. The remaining $3$ cases without convergence occur under the low-frequency, high-voltage and the low-pressure conditions, suggesting that a longer integration duration is necessary. Within this sweep, the time step $\Delta t$ is adaptively configured based on $f_{\mathrm{RF}}$ and $p$ instead of being fixed, ranging from $9.2183 \times 10^{-11}$ to $7.3746 \times 10^{-10}~\mathrm{s}$. Specifically, The maximum verified reliable time step decreases as the frequency increases, dropping from $2.00 \times 10^{-10}~\mathrm{s}$ (at $6.78--13.56~\mathrm{MHz}$) to $9.2183 \times 10^{-11}~\mathrm{s}$ (at $54.24~\mathrm{MHz}$). Overall, \FIPSGEFC demonstrates exceptional numerical robustness across the majority of the tested operating space.

\begin{figure}
\centering
\includegraphics[width=.65\linewidth]{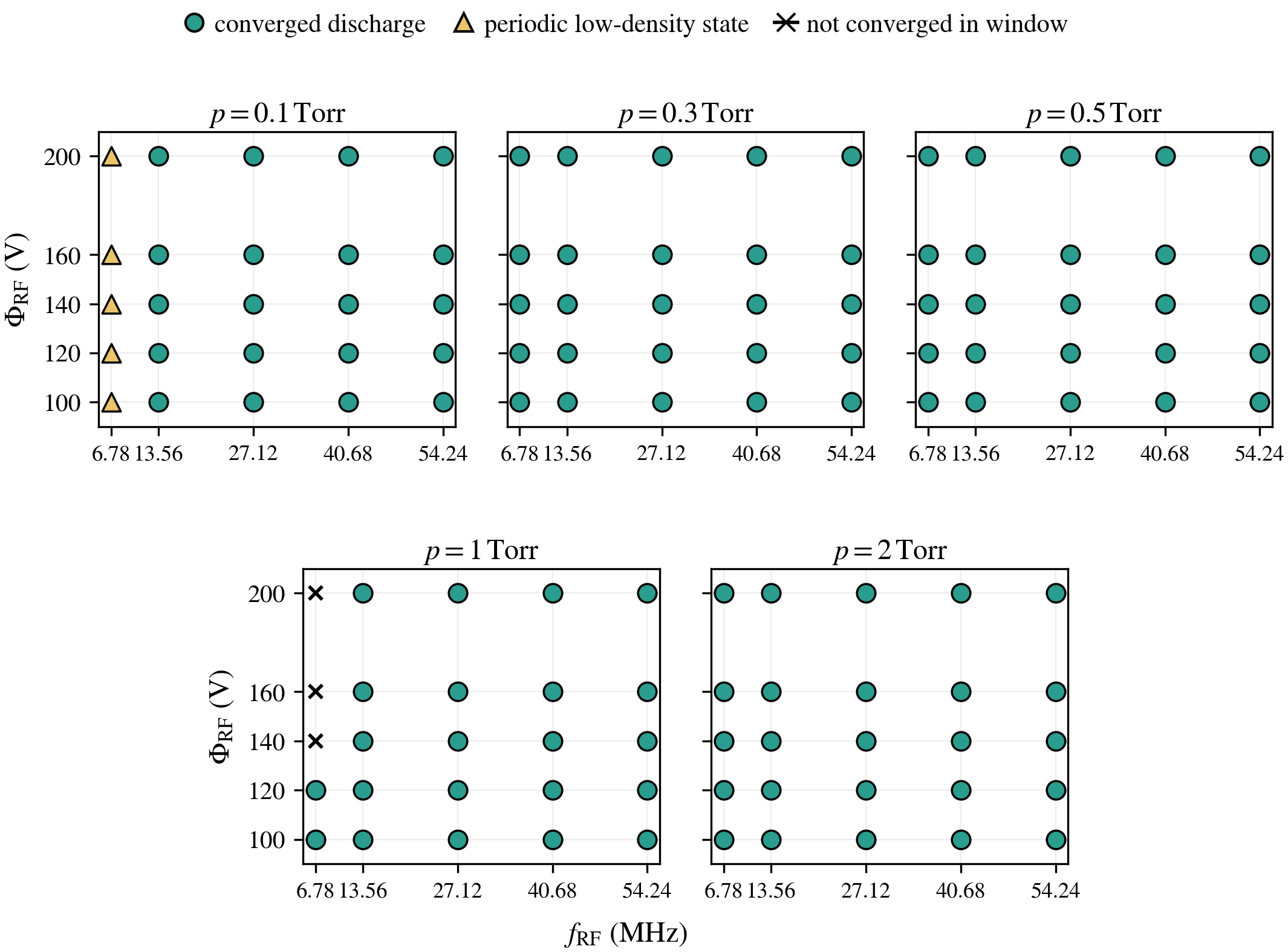}
\caption{Periodic-state classification of \FIPSGEFC over 125 pressure--voltage--frequency combinations}
\label{fig:robustness_map}
\end{figure}

Numerical convergence alone does not guarantee a physically meaningful discharge. \cref{fig:robustness_profiles} presents the cycle-averaged profiles of $n_e$, $n_+$, $T_e$ and $\phi$ at $p \in \{0.3, 0.5, 1, 2\}~\mathrm{Torr}$ for $\Phi_{\mathrm{RF}} = 140~\mathrm{V}$ and $f_{\mathrm{RF}} = 13.56~\mathrm{MHz}$. Across all the cases, the solutions capture the characteristic bulk-plasma and sheath structures consistently. These profiles replicate the expected spatial organization of a CCP discharge well, further verifying the physical validity and the robustness of \FIPSGEFC.

\begin{figure}
\centering
\includegraphics[width=.65\linewidth]{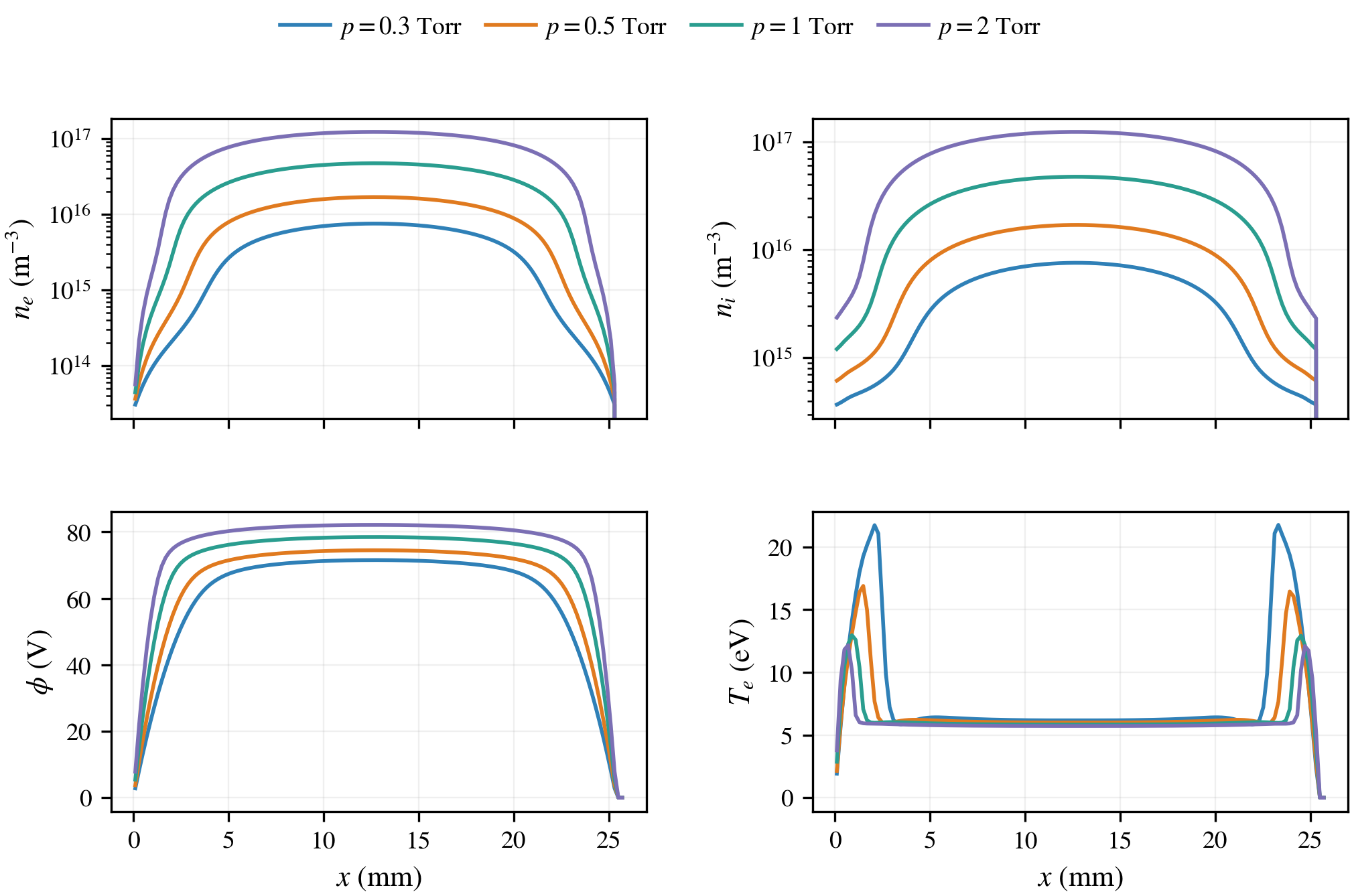}
\caption{Cycle-averaged discharge profiles produced by \FIPSGEFC at $4$
pressures with $\Phi_{\mathrm{RF}}=140~\mathrm{V}$ and
$f_{\mathrm{RF}}=13.56~\mathrm{MHz}$}
\label{fig:robustness_profiles}
\end{figure}

To bridge the long-term simulations with the sufficient conditions derived in \cref{sec:analysis}, phase-dependent bounds are evaluated over a single RF cycle at the $500$-th, $1000$-th and the $1500$-th cycles, as shown in \cref{fig:theory_dt_bridge}. The left and right panels correspond to \eqref{eq:t1}-\eqref{eq:t3}, respectively, while results for \eqref{eq:t4} are omitted. For the electron energy equation \eqref{eq:3}, only the interior $t_2$ bound is plotted, excluding the additional condition for the uniform bound. The horizontal dotted line marks the maximum verified reliable time step ($\Delta t = 9.60238 \times 10^{-11}~\mathrm{s}$). At these three checkpoints, the minimum density-related bounds are $1.2088 \times 10^{-10}$, $1.2026 \times 10^{-10}$ and $1.2011 \times 10^{-10}~\mathrm{s}$, whereas the corresponding bounds for the electron energy are $1.6094 \times 10^{-11}$, $1.5687 \times 10^{-11}$ and $1.5592 \times 10^{-11}~\mathrm{s}$. The nearly overlapping curves confirm the approach to a periodic steady state. Notably, the density-related bounds remain comparable to the empirical reliable time step, whereas the bound for the electron energy is more restrictive, with its minimum at the $1500$-th cycle being approximately $16.1\%$ of the empirical value. This discrepancy highlights the conservative nature of the sufficient conditions and should not be misconstrued as a sharp prediction of the actual stability limit or an independent validation of the theoretical results.

\begin{figure}
\centering
\includegraphics[width=.65\linewidth]{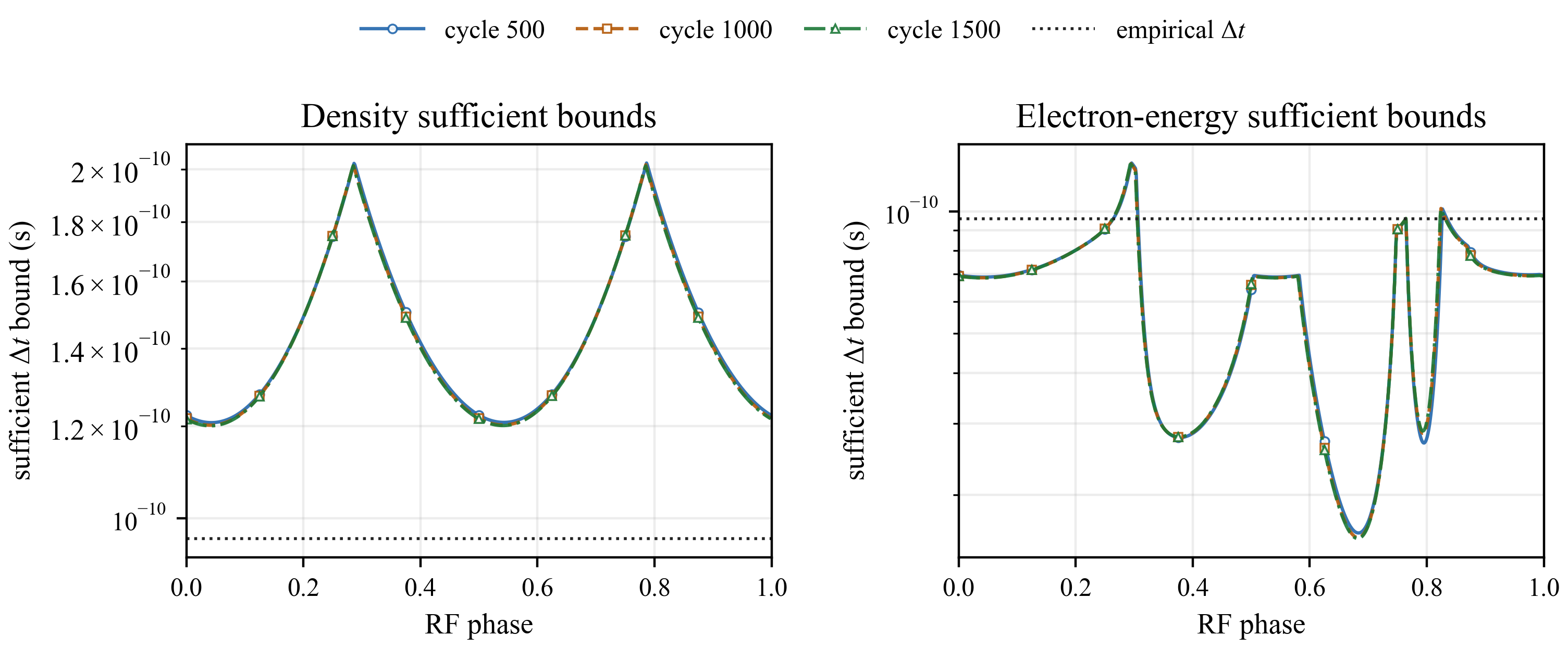}
\caption{Phase-dependent evaluated sufficient time-step bounds for \FIPSGEFC
at RF cycles $500, 1000, 1500$}
\label{fig:theory_dt_bridge}
\end{figure}

\subsection{Temporal and spatial accuracy}
\label{sec:accuracy}
Numerical robustness alone does not guarantee the phase-resolved accuracy. Therefore, the temporal and spatial convergence are assessed exclusively for \FIPSGEFC. 
For temporal refinement, the mesh size is fixed at $N=128$, while the number of time steps per RF cycle varies within $N_{\mathrm{RF}}=768, 1536, 3072, 6144$. 
Each case is initialized from the same periodic steady state. The weighted relative $L^2$ errors for $n_e$, $n_{+}$, $T_e$ and $\phi$ are evaluated at four identical RF phases, relative to the $N_{\mathrm{RF}}=6144$ baseline. As shown in \cref{fig:temporal_refinement}, the error decreases with smaller $\Delta t$ (larger $N_{\mathrm{RF}}$), aligning well with the first-order reference slope $\mathcal{O}(\Delta t)$.

All errors decrease consistently with time-step refinement. At phase zero, increasing $N_{\mathrm{RF}}$ from $768$ to $3072$ reduces the errors in $n_e$, $n_{+}$, $T_e$ and $\phi$ from $0.106\%$, $0.0930\%$, $0.1235\%$ and $0.560\%$ to $0.0152\%$, $0.0133\%$, $0.0176\%$ and $0.0813\%$, respectively. Over the finer intervals, $n_e$ and $n_{+}$ exhibit a first-order temporal convergence approximately. In contrast, the errors in $T_e$ and $\phi$ depend more strongly on both the RF phase and the refinement level. Consequently, while larger time steps suffice for the long-time stable integration, capturing accurate phase-resolved $T_e$ and $\phi$ profiles demands finer temporal resolution.

\begin{figure}
\centering
\includegraphics[width=.65\linewidth]{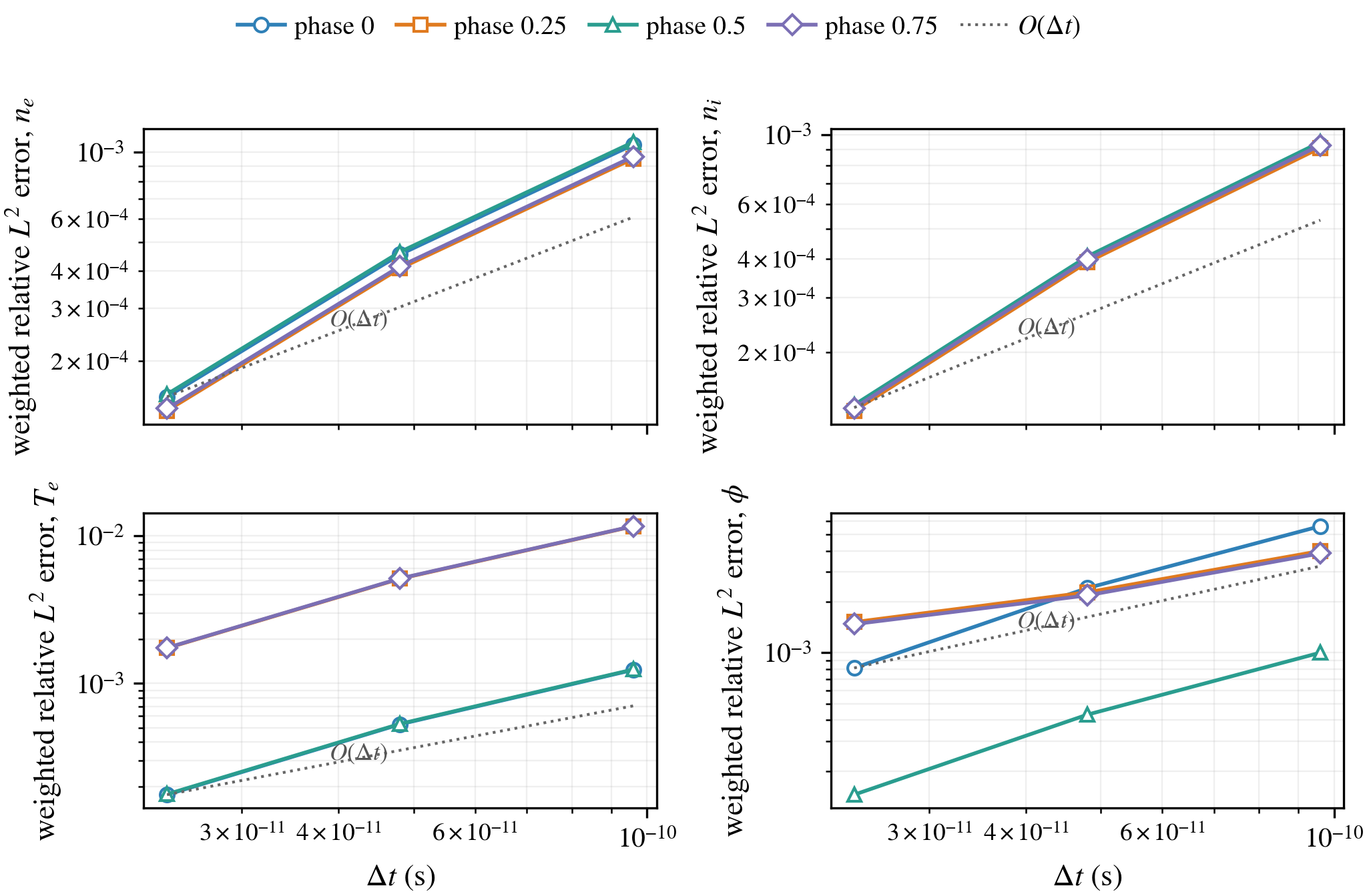}
\caption{Phase-resolved temporal-refinement errors of \FIPSGEFC on the
$N=128$ mesh}
\label{fig:temporal_refinement}
\end{figure}

To verify that $1500$ RF cycles yield a representative long-time state, the main-case simulation with $N_{\mathrm{RF}}=768$ is extended to $5000$ cycles. The relative differences in cycle-averaged $n_e, n_{+}, T_e, \phi$ are $0.6509\%$, $0.6451\%$, $0.0412\%$ and $0.0568\%$, respectively. These results support the use of $1500$ RF cycles as a common integration window for scheme comparisons in this benchmark.

For spatial refinement, we fix the temporal resolution at $N_{\mathrm{RF}}=2048$ while the uniform mesh size varies ($N=128, 256, 512, 1024$), with $N=1024$ serving as the reference. Coarser solutions are linearly interpolated to the reference cell centers for comparison. As shown in \cref{fig:spatial_refinement}, all errors decrease monotonically with the cell width $\Delta x$. From $N=256$ to $512$, the spatial convergence orders for $n_e$, $n_{+}$, $T_e$ and $\phi$ are approximately $1.68$, $1.68$, $0.94$ and $1.63$, respectively. The lower order of convergence for $T_e$ is consistent with its heightened sensitivity to sheath and boundary-layer discretization.

\begin{figure}
\centering
\includegraphics[width=.65\linewidth]{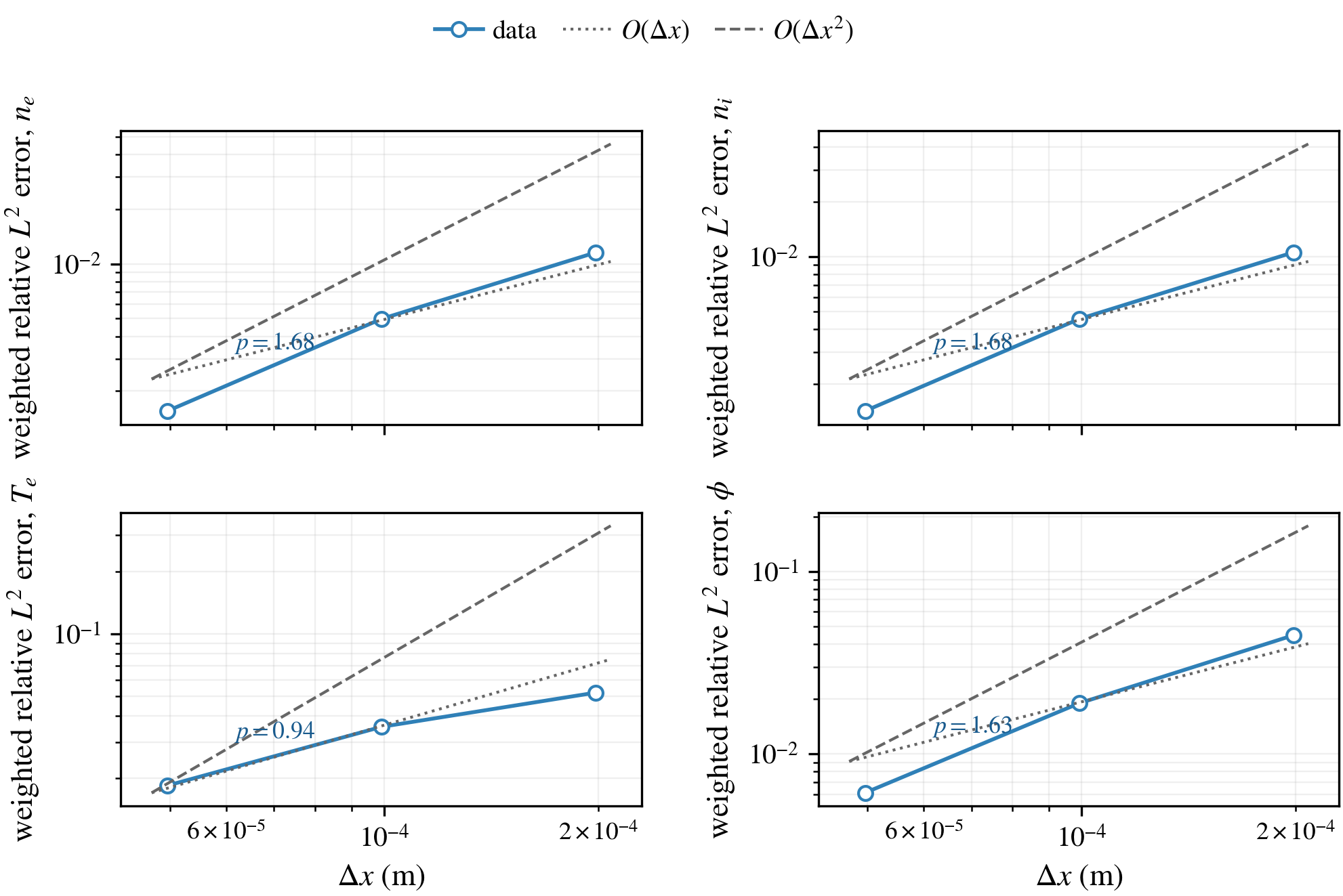}
\caption{Spatial-refinement errors of \FIPSGEFC for
$N_{\mathrm{RF}}=2048$}
\label{fig:spatial_refinement}
\end{figure}

\subsection{Computational efficiency}
\label{sec:efficiency}
While \FIPSGEFC is used for the preceding analysis because of its complete discrete structure and mathematical properties, analytical tractability does not imply runtime optimality. Therefore, we compare \FEP, \FIP, \FIPSG and \FIPSGEBC in terms of their maximum verified reliable time steps, wall-clock times over $1500$ RF cycles and speedups relative to \FEP. Here, the \emph{maximum verified reliable time step} is the largest tested $\Delta t$ that maintains a physically admissible state throughout the prescribed integration interval. It is an empirical quantity rather than a theoretical continuous stability limit.

All four comparison schemes use identical physical, grid and solver configurations. Each timing measurement is repeated ten times. The median wall-clock time over $1500$ RF cycles is reported, with the first ($Q1$) and the third ($Q3$) quartiles quantifying the run-to-run variation. The results are summarized in \cref{tab:reference_efficiency} and \cref{fig:reference_efficiency}. The median runtimes for \FEP, \FIP, \FIPSG and \FIPSGEBC are $77.32$~s, $30.56$~s, $19.45$~s and $14.56$~s, respectively, corresponding to speedups of $1.00$, $2.53$, $3.98$ and $5.31$. Their corresponding maximum verified reliable time steps are $8.23\times10^{-11}$~s, $1.15\times10^{-9}$~s, $2.30\times10^{-9}$~s and
$3.07\times10^{-9}$~s.

Numerical results demonstrate the successive benefits of the implicit electron energy advancement, the SG transport scheme and the energy-boundary treatment. Although implicit formulations entail higher computational overhead per iteration, their significantly larger permissible time steps heavily reduce the total time-to-solution, which serves as the primary metric for the long-time CCP simulations. Moreover, in the practical implementation, the computational efficiency of \FIPSGEBC can be further enhanced by employing an approximate implicit framework that omits the chemical source term $S_{e}^{n}$ within the semi-implicit Poisson solver.

\begin{table}
\centering
\caption{Repeated timing statistics for different numerical schemes over $1500$ RF cycles}
\label{tab:reference_efficiency}
\small
\begin{tabular}{lrrrr}
\toprule
Scheme & Median time (s) & Q1 (s) & Q3 (s) & Speedup relative to \FEP \\
\midrule
\FEP & 77.32 & 76.99 & 77.63 & 1.00 \\
\FIP & 30.56 & 30.35 & 31.01 & 2.53 \\
\FIPSG & 19.45 & 19.39 & 19.53 & 3.98 \\
\FIPSGEBC & 14.56 & 14.52 & 14.60 & 5.31 \\
\bottomrule
\end{tabular}
\end{table}

\begin{figure}
\centering
\includegraphics[width=.65\linewidth]{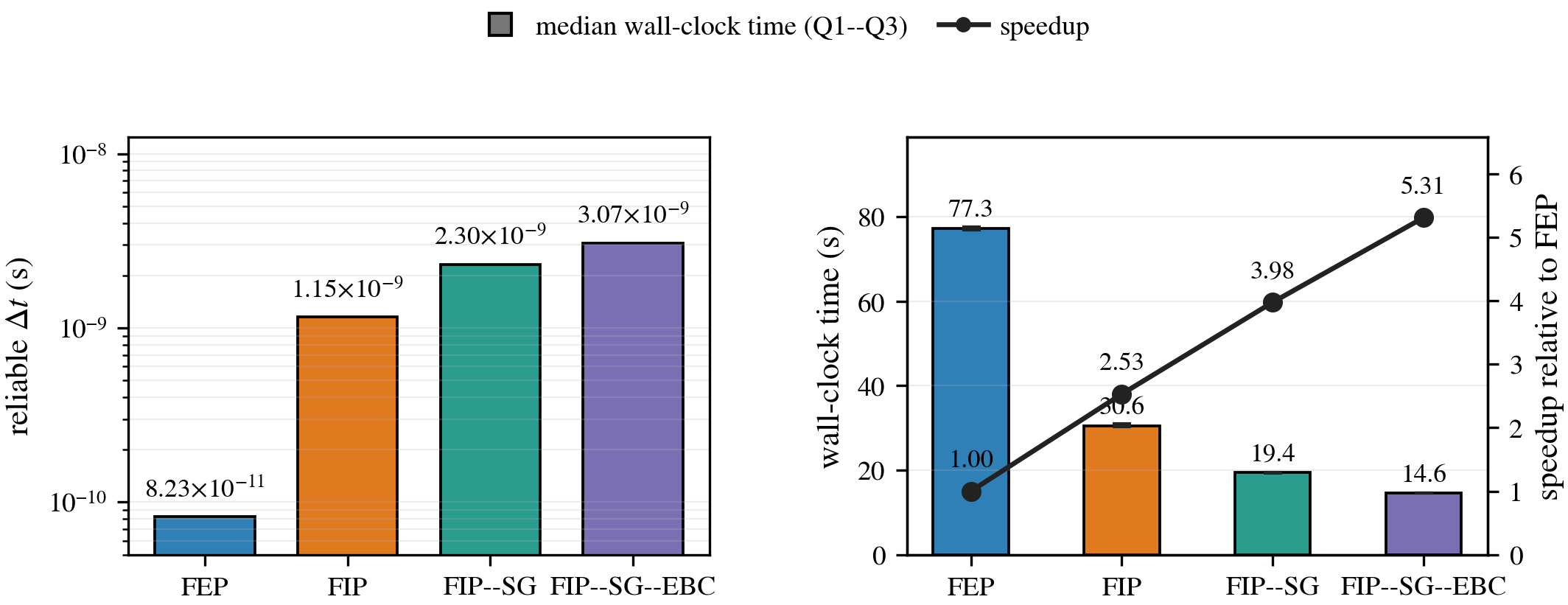}
\caption{Verified reliable time steps and repeated timing results for different numerical schemes}
\label{fig:reference_efficiency}
\end{figure}

To further examine the influence of the chemical source term in the Poisson predictor on the computational efficiency, we compare the SG-based schemes with and without this term. In the source-off settings, the lagged chemical source $S_e^n$ is omitted only from the SG-Picard Poisson predictor, while the chemical source remains active in the subsequent implicit fluid update. Therefore, the underlying physical model is unchanged. Numerical experiments are shown in \cref{tab:source_ablation}.

For \FIPSG and \FIPSGEBC, removing the chemical source from the Poisson predictor does not change the verified time-step settings, which remain $N_{\mathrm{RF}}=32$ and $N_{\mathrm{RF}}=24$, respectively. Under the same settings, the median runtime of \FIPSG decreases from $19.45$~s to $17.70$~s, while that of \FIPSGEBC decreases from $14.56$~s to $13.33$~s, corresponding to reductions of approximately 8.97\% and 8.44\%, respectively. A similar reduction is observed for \FIPSGEFC, for which the median runtime decreases from $487.16$~s to $443.27$~s, corresponding to a reduction of 9.01\%.

\begin{table}
\centering
\caption{Effect of the chemical source term in the Poisson predictor on the computational time of the SG-based schemes}
\label{tab:source_ablation}
\begin{tabular}{lccccc}
\hline
Scheme & Predictor source & $N_{\mathrm{RF}}$ & Runs & Time (s) & Q1--Q3 (s) \\
\hline
\FIPSG    & on  & 32  & 10 & 19.45 & 19.39--19.53 \\
\FIPSG    & off & 32  & 10 & 17.70 & 17.68--17.77 \\
\FIPSGEBC & on  & 24  & 10 & 14.56 & 14.52--14.60 \\
\FIPSGEBC & off & 24  & 10 & 13.33 & 13.30--13.43 \\
\FIPSGEFC & on  & 768 & 10 & 487.16 & 486.69--487.98 \\
\FIPSGEFC & off & 768 & 10 & 443.27 & 442.91--443.43 \\
\hline
\end{tabular}
\end{table}

\section{Conclusions}
\label{sec:conclusions}
This work investigates the numerical schemes of the 1D-CCP. We introduce a new numerical scheme that combines a semi-implicit solver for the Poisson's equation with the Scharfetter-Gummel discretizations for the continuity and the electron energy equations, which can also be extended to the 2D cases. Several schemes are compared in terms of the efficiency, accuracy and the robustness. Among them, a hybrid scheme with different schemes for different equations provides the best overall performance. Some details in the implementation are also discussed.

\section*{Acknowledgement}
\label{sec:acknowledgement}
The authors declare that they have no conflict of interest. Authors are listed in alphabetical order. We thank Qian Yin from the Hong Kong Polytechnic University for the discussion about this topic.

\bibliography{references}
\end{document}